\documentclass{article}
\usepackage{amsmath,amssymb,amsthm,amsfonts}
\usepackage{fix-cm}
\usepackage{makeidx}
\usepackage{amscd}
\usepackage{tikz}
\usepackage{mathrsfs}
\usepackage{graphicx}
\usepackage{epstopdf}
\usepackage{setspace} %% [nodisplayskipstretch]
\usepackage{mdwlist}
\usepackage{float}
\usepackage{keyval}
\usepackage{caption}
\usepackage{subcaption}
\usepackage[english]{babel}
\usepackage{color}

\newtheorem{theorem}{Theorem}[section]
\newtheorem{proposition}[theorem]{Proposition}
\newtheorem{corollary}[theorem]{Corollary}
\newtheorem{lemma}[theorem]{Lemma}
\newtheorem{question}[theorem]{Question}

\theoremstyle{definition}
\newtheorem{definition}[theorem]{Definition}

\newcommand\R{\mathbb R}
\newcommand\Rp{\mathbb{R}^+}
\newcommand\N{\mathbb N}
\newcommand\No{{\mathbb{N}_0}}
\newcommand\C{\mathbb C}
\newcommand{\abs}[1]{\lvert #1\rvert}

\newcommand\dist{\mathrm{dist}}
\newcommand\ob[2]{{\rm B}(#1,#2)}
\newcommand\cb[2]{\bar{\rm B}(#1,#2)}
\newcommand\ad{\delta_{\rm ann}}
\newcommand\ar{\rho_{\rm ann}}
\newcommand\ec[1]{\varepsilon_{#1}}

\newcommand\scs{\mathcal F}
\newcommand\noninc{\mathcal N}
\newcommand\La[2][{A}]{\mathcal{L}_{#1}(\mathcal{#2})}
\newcommand\LM[2][{A}]{\mathcal{L}_{#1}^*(\mathcal{#2})}
\newcommand\An{\mathcal A}
\newcommand\Sw[1][{A}]{\mathcal S(#1)}

\DeclareMathOperator{\tint}{int}
\newcommand{\sskip}{\vskip 0.2cm}

\begin{document}

%\begin{frontmatter}
\title{Abstract Swiss Cheese Space and Classicalisation of Swiss Cheeses}
\author{%
J. F. Feinstein
%\\
%School of Mathematical Sciences, University of Nottingham\\
%University Park, Nottingham, NG7 2RD, UK\\
%E-mail: joel.feinstein@nottingham.ac.uk
\and%
S. Morley\footnote{This author is supported by a grant from the EPSRC}
%\\
%School of Mathematical Sciences, University of Nottingham\\
%University Park, Nottingham, NG7 2RD, UK\\
%E-mail: pmxsm9@nottingham.ac.uk
\and%
H. Yang\footnote{This author is supported by a China Tuition Fee Research Scholarship and the School of Mathematical Sciences at the University of Nottingham}
%School of Mathematical Sciences, The University of Nottingham\\
%University Park, Nottingham, NG7 2RD, UK
%E-mail: pmxhy1@nottingham.ac.uk
}

\maketitle
\renewcommand{\thefootnote}{}

\footnote{2010 \emph{Mathematics Subject Classification}: Primary 46J10; Secondary 54H99.}

\footnote{\emph{Key words and phrases}: Swiss cheeses, rational approximation, uniform algebras, bounded point derivations, regularity of $R(X)$.}

\renewcommand{\thefootnote}{\arabic{footnote}}
\setcounter{footnote}{0}

\begin{abstract}
  Swiss cheese sets are compact subsets of the complex plane obtained by deleting a sequence of open disks from a closed disk. Such sets have provided numerous counterexamples in the theory of uniform algebras. In this paper, we introduce a topological space whose elements are what we call ``abstract Swiss cheeses''. Working within this topological space, we show how to prove the existence of ``classical'' Swiss cheese sets (as discussed in \cite{feinheath2010}) with various desired properties.

  We first give a new proof of the Feinstein-Heath classicalisation theorem (\cite{feinheath2010}). We then consider when it is possible to ``classicalise'' a Swiss cheese while leaving disks which lie outside a given region unchanged. We also consider sets obtained by deleting a sequence of open disks from a closed annulus, and we obtain an analogue of the Feinstein-Heath theorem for these sets. We then discuss regularity for certain uniform algebras. We conclude with an application of these techniques to obtain a classical Swiss cheese set which has the same properties as a non-classical example of O'Farrell \cite{o1979regular}.
\end{abstract}

%\begin{keyword}

%\MSC[2010] Primary 46J10\sep Secondary 54H99
%\end{keyword}

%\end{frontmatter}

%\linenumbers

\section{Introduction}
Throughout, we use the term {\em compact plane set} to mean a non-empty, compact subset of the complex plane. Let $X$ be a compact plane set. Then $C(X)$ denotes the set of all continuous, complex-valued functions on $X$, and $R(X)$ denotes the set of those functions $f\in C(X)$ which can be uniformly approximated on $X$ by rational functions with no poles on $X$. Both $R(X)$ and $C(X)$ are uniform algebras on $X$. We refer the reader to \cite{browder1969,dales2000,gamelin1984} and \cite{stout1971} for further definitions and background concerning uniform algebras and Banach algebras.

A Swiss cheese set is a compact subset of $\C$ obtained by deleting a sequence of open disks from a closed disk. Such sets have been used as examples in the theory of uniform algebras and rational approximation. Swiss cheese sets were introduced by Roth \cite{roth1938}, where she gave the first known example of a compact plane set $X$ such that $R(X)\neq C(X)$ but $X$ has empty interior. Since then there have been numerous applications of Swiss cheese sets in the literature.

One notable example of a Swiss cheese construction is due to McKissick \cite{mckissick1963nontrivial}. He gave an example of a Swiss cheese set $X$ such that $R(X)$ is regular but $R(X)\neq C(X)$. (We will define regularity in Section \ref{regularitysection}.) The sequence of open disks used to construct this Swiss cheese set may touch or overlap, which means that the set $X$ might have undesirable topological properties. To improve the topological properties of the resulting Swiss cheese set, while preserving the properties of the uniform algebra, a process that we call {\em classicalisation} was developed (\cite{feinheath2010}).

We may consider a pair consisting of a closed disk and a collection of open disks in the plane, from which we obtain the desired Swiss cheese set (see Definition \ref{swisscheedef} below). We call such a pair a {\em Swiss cheese} and say it is {\em classical} if the collection of open disks and the complement of the closed disk have pairwise disjoint closures and the sum of the radii of all open disks is finite. Note that, in the literature, the term `Swiss cheese' traditionally refers to what we call a Swiss cheese set. Feinstein and Heath \cite{feinheath2010} considered Swiss cheeses in which the sum of the radii of the open disks is strictly less than the radius of the larger, closed disk. They proved, using Zorn's lemma, that for such a Swiss cheese, the associated Swiss cheese set contains a Swiss cheese set associated to a classical Swiss cheese. Later, Mason \cite{mason2010} gave a proof of this theorem using transfinite induction.

Classical Swiss cheese sets have many desirable topological properties. For example, Dales and Feinstein \cite{dalefein2010} proved that given two points $x,y$ in a classical Swiss cheese set there is a rectifiable path connecting $x,y$ and such that the length of this path is no more than $\pi\abs{x-y}$; in fact, the constant $\pi$ can be replaced by $\pi/2$ here. After this observation it is easy to see that a classical Swiss cheese set is path connected (and hence connected), locally path connected (and hence locally connected), and uniformly regular, as defined in \cite{dalefein2010}.  Also as a consequence of connectedness, we see that a classical Swiss cheese set cannot have any isolated points. In \cite{feinheath2010} it was noted that every classical Swiss cheese set with empty interior is homeomorphic to the Sierpi\'nski carpet as a consequence of a theorem of Whyburn \cite{whyburn1958}.

Browder \cite{browder1969} notes that if $X$ is a classical Swiss cheese set then $R(X)$ is essential (see also \cite{feinheath2010}). In particular, $R(X)\neq C(X)$, as originally proved by Roth \cite{roth1938}. It follows from the Hartogs-Rosenthal theorem that $X$ must have positive area. A direct proof that every classical Swiss cheese set has positive area is due to Allard, as outlined in \cite[pp.~163-164]{browder1969}.

Where existing examples of Swiss cheese sets in the literature are not classical, it is of interest to construct classical Swiss cheese sets which solve the same problems. As part of a general classicalisation scheme, we discuss some new techniques for constructing such classical Swiss cheese sets.

In this paper we consider what we call {\em abstract Swiss cheeses}, which are sequences of pairs consisting of a complex number and a non-negative real number. Each pair in this sequence corresponds to a centre and radius of a disk in the plane. We give the set of all abstract Swiss cheeses a natural topology and use this topology to give a new proof of the Feinstein-Heath theorem. We show that, under some conditions, we can classicalise Swiss cheese sets while only changing open disks which lie in certain regions. We prove an analogue of the Feinstein-Heath theorem for annuli. We give some results regarding regularity of $R(X)$ for unions of compact plane sets, which will be used in the final section. Finally, we give an example of the application of a combination of these results to construct an example of a classical Swiss cheese set $X$ such that $R(X)$ is regular and admits a non-degenerate bounded point derivation of infinite order (as defined in Section \ref{applicationoflocal}), which improves an example of O'Farrell \cite{o1979regular}. This fits into our general classicalisation scheme.

\section{Swiss cheeses and abstract Swiss cheese space}
We denote the set of all non-negative real numbers by $\Rp,$ the set of positive integers by $\N$ and the set of all non-negative integers by $\No$. Let $a\in\C$ and let $r>0$. We denote the open disk of radius $r$ and centre $a$ by $\ob ar$ and the corresponding closed disk by $\cb ar$. We also set $\cb{a}{0}=\{a\}$ and $\ob{a}{0}=\emptyset$. We say a disk with radius zero is {\em degenerate}. For a non-degenerate open or closed disk $D$ in the plane, let $r(D)$ denote the radius of $D$; for a degenerate disk $D$ we define $r(D)=0$. The following is the definition of a Swiss cheese used in \cite{feinheath2010}.

\begin{definition}\label{swisscheedef}Let $\Delta\subseteq\C$ be a non-degenerate open disk and let $\mathcal D$ be a countable collection of non-degenerate$,$ open disks in the plane. Then the ordered pair $E=(\overline\Delta,\mathcal D)$ is a {\em Swiss cheese}. We also define the following.
\begin{enumerate}
 \item The {\em Swiss cheese set $X_E$ associated with the Swiss cheese $E$} is defined by
 \begin{equation}\label{Scset}
 X_E=\overline\Delta\setminus\bigcup\limits_{D\in\mathcal D}D.
 \end{equation}
 \item The {\em discrepancy} $\delta(E)$ of $E$ is defined by
 \[
 \delta(E)=r(\overline\Delta)-\sum_{D\in\mathcal D}r(D).
 \]
 \item The Swiss cheese $E$ is {\em semiclassical} if $\delta(E)>-\infty$, for each $D\in\mathcal D$ we have $D\subseteq\overline\Delta$, and for each $D'\in\mathcal D$ with $D\neq D'$ we have $D\cap D'=\emptyset$. In this case we say the Swiss cheese set associated to $E$ is {\em semiclassical}.
 \item The Swiss cheese $E$ is {\em classical} if $\delta(E)>-\infty$, for each $D\in\mathcal D$ we have $\overline D\subseteq\Delta$, and for each $D'\in\mathcal D$ with $D\neq D'$ we have $\overline D\cap \overline{D'}=\emptyset$. In this case we say the Swiss cheese set associated to $E$ is {\em classical}.
 \item The Swiss cheese $E$ is {\em finite} if the collection $\mathcal D$ is finite and infinite otherwise.
\end{enumerate}
\end{definition}

The condition $\delta(E)>-\infty$ is equivalent to the sum of the radii of the open disks being finite.

We note that without some condition on the disks in $\mathcal D$ we can obtain every compact plane set as a Swiss cheese set with this definition.
\sskip
Throughout this paper, we will work in what we call {\em abstract Swiss cheese space} $\scs$, where $\scs=(\C\times\Rp)^{\No}$ with the product topology.

\begin{definition}\label{absSCdef}
Let $A=((a_n,r_n))_{n=0}^\infty\in\scs$. We call $A$ an {\em abstract Swiss cheese}$,$ and we define the following.
\begin{enumerate}
 \item The {\em significant index set of $A$} is $S_A:=\{n\in\N:r_n>0\}$. We say that $A$ is {\em finite} if $S_A$ is a finite set, otherwise $A$ is {\em infinite}.
 \item The {\em associated Swiss cheese set} $X_A$ is defined by
 \begin{equation}\label{absScs}
 X_A=\cb{a_0}{r_0}\setminus\left( \bigcup\limits_{n=1}^\infty\ob{a_n}{r_n}\right).
 \end{equation}
 \item We say that $A$ is {\em semiclassical} if $\sum_{n=1}^\infty r_n<\infty,$ $r_0>0$ and for all $k\in S_A$ the following hold$:$
     \begin{enumerate}
        \item $\ob{a_k}{r_k}\subseteq\ob{a_0}{r_0};$
        \item whenever $\ell\in S_A$ has $\ell\neq k,$ we have $\ob{a_k}{r_k}\cap\ob{a_\ell}{r_\ell}=\emptyset$.
     \end{enumerate}
 \item We say that $A$ is {\em classical} if $\sum_{n=1}^\infty r_n<\infty,$ $r_0>0$ and for all $k\in S_A$ the following hold$:$
    \begin{enumerate}
        \item $\cb{a_k}{r_k}\subseteq\ob{a_0}{r_0};$
        \item whenever $\ell\in S_A$ with $\ell\neq k,$ we have $\cb{a_k}{r_k}\cap\cb{a_\ell}{r_\ell}=\emptyset$.
    \end{enumerate}
\end{enumerate}
For $\alpha\geq 1$ we define the {\em discrepancy function of order} $\alpha$ , $\delta_\alpha:\scs\to[-\infty,\infty)$ by
\begin{equation}\label{gendiscr}\delta_\alpha(A)=r_0^\alpha -\sum\limits_{n=1}^\infty r_n^\alpha \qquad (A=((a_n,r_n))_{n=0}^\infty\in \mathcal F).
\end{equation}
%for $A=((a_n,r_n))_{n\geq 0}\in\scs$.
\end{definition}

Note that in \eqref{absScs} we could instead write \[X_A:=\cb{a_0}{r_0}\setminus\left( \bigcup\limits_{n\in S_A}\ob{a_n}{r_n}\right).\]
If $A$ is semiclassical or classical, then $\pi\delta_2(A)$ is the area of the Swiss cheese set $X_A$.
We will usually write $A=((a_n,r_n))$ for an abstract Swiss cheese. All sequences, unless otherwise specified, will be indexed by $\No$.

We also define the following functions on $\scs$.

\begin{definition}
The {\em radius sum function} is the map $\rho:\scs\to[0,\infty]$ defined by
\[
\rho(A)=\sum\limits_{n=1}^\infty r_n\qquad(A=((a_n,r_n))\in\scs).
\]
The {\em centre bound function} is the map $\mu:\scs\to[0,\infty]$ defined by
\[
\mu(A)=\sup_{n\in\N}{\abs{a_n}}\qquad(A=((a_n,r_n))\in\scs).
\]
Let $E\subseteq\C$. For an abstract Swiss Cheese $A=((a_n,r_n))$ we define $H_A(E)$ to be the set of those $n\in S_A$ such that $\cb{a_n}{r_n}\cap E\neq\emptyset$. The {\em local radius sum function on $E$} is the function $\rho_E:\scs\to[0,\infty]$ defined by
\[
\rho_E(A)=\sum\limits_{n\in H_A(E)}r_n\qquad(A=((a_n,r_n))\in\scs).
\]
\end{definition}

It is easy to see that $\rho$ and $\mu$ are both lower semicontinuous from $\scs$ to $[0,\infty]$. (For $\rho$, this is an easy consequence of Fatou's lemma for series.)
\sskip
We now explain the connection between Swiss cheeses, as in Definition \ref{swisscheedef}, and abstract Swiss cheeses. We construct a many-to-one surjection of a subset of $\scs$ onto the collection of all Swiss cheeses as in Definition \ref{swisscheedef}. Let $A=((a_n,r_n))$ be an abstract Swiss cheese with $r_0>0$. Then we can obtain an associated Swiss cheese $E_A$ by setting
\[
E_A:=(\cb{a_0}{r_0},\{\ob{a_n}{r_n}:n\in S_A\}).
\]
The associated Swiss cheese sets of $A$ and $E_A$ are equal, and $\delta(E_A)\geq\delta_1(A)$. Moreover, if $A$ is finite then $E_A$ is finite; if $A$ is semiclassical then $E_A$ is semiclassical; and if $A$ is classical then $E_A$ is classical. Conversely, if $E$ is a finite Swiss cheese then there is a finite abstract Swiss cheese $A$ such that $E_A=E$.

Let $E=(\overline\Delta,\mathcal D)$ be a Swiss cheese. If $E$ is (semi)classical then there is an abstract Swiss cheese $A$ with $E_A=E$ such that $A$ is (semi)classical. Moreover, when the sum of the radii of open disks in $\mathcal D$ is finite, we can find an abstract Swiss cheese $A=((a_n,r_n))$ with $\rho(A)<\infty$ and $E=E_A$ such that the sequence $(r_n)_{n=1}^\infty$ is non-increasing.

We denote the collection of all abstract Swiss cheeses $A=((a_n,r_n))$ with $\rho(A)<\infty$ and $(r_n)_{n=1}^\infty$ non-increasing by $\noninc$. In addition, for each $M>0$ and $R>0$, we denote the set of all those abstract Swiss cheeses $A=((a_n,r_n))\in\noninc$ such that $\mu(A)\leq M$ and $\rho(A)\leq R$ by $\noninc(M,R)$. {Let $M,R>0$. Although $\noninc(M,R)$ is not itself compact, a closed subset $S$ of $\noninc(M,R)$ is compact if and only if the $0$-th coordinate projection maps $S$ to a bounded subset of $\C\times \Rp$}. Note that, for $A=((a_n,r_n))\in\noninc(M,R)$ we have $r_n\leq R/n$ for all $n\in\N$.

\begin{lemma} \label{discrepancy_function}
Let $M,R>0$. For $\alpha\geq 1,$ the function $\delta_\alpha:\scs \to[-\infty,\infty)$ is upper semicontinuous. For $\alpha>1,$ the function $\delta_\alpha|_{\noninc(M,R)}:\noninc(M,R)\to \R$ is continuous.
\end{lemma}
\begin{proof}
As for the lower semicontinuity of $\rho$, it is an easy consequence of Fatou's lemma for series that $\delta_\alpha:\scs\to[-\infty,\infty)$ is an upper semicontinuous function for each $\alpha\geq 1$.

Fix $\alpha>1$. For each $m\in\No$ let  $A^{(m)}=((a_n^{(m)},r_n^{(m)}))\in\noninc(M,R)$ and suppose $A^{(m)}\to A\in\noninc(M,R)$ as $m\to\infty$. We have $\abs{r_n^{(m)}}^{\alpha}\leq {R^{\alpha}}/{n^{\alpha}}$ for all $n\in\N$. Since $\sum_{n=1}^\infty {R^{\alpha}}/{n^{\alpha}}<\infty$, by the dominated convergence theorem, we have
\[
\delta_\alpha (A) = r_0^\alpha-\sum_{n=1}^\infty r_n^{\alpha} = \lim_{m\to \infty} \left( (r_0^{(m)})^{\alpha}- \sum_{n=1}^\infty(r_n^{(m)})^{\alpha}\right) = \lim_{m\to \infty} \delta_\alpha(A^{(m)}).
\]
So $\delta_\alpha$ is continuous from $\noninc(M,R)$ to $\R$.
\end{proof}

We remark that there are examples showing that $\delta_1$ is only upper semicontinuous, but not continuous.

\begin{definition} \label{disk_mani}
Let $A=((a_n,r_n))$ be an abstract Swiss cheese.
\begin{enumerate}
  \item Let $a\in\C$ and $r> 0$ and let $m\in \No$. We say an abstract Swiss cheese $B=((b_n,s_n))$ is obtained from $A$ by {\em inserting} a disk $\ob ar$ at index $m$ if, for $0\leq n<m$, we have $b_n=a_n,$ $s_n=r_n$; for $n>m$ we have $b_n=a_{n-1},$ $s_n=a_{n-1}$, and $b_m=a$, $s_m=r$.
  \item Let $m\in\No$. We say an abstract Swiss cheese $B=((b_n,s_n))$ is obtained from $A$ by {\em deleting} the disk at index $m$ if, for $0\leq n<m$, we have $b_n=a_n,$ $s_n=r_n$ and for all $n\geq m$ we have $b_n=a_{n+1},$ $s_n=r_{n+1}$.
  \item Suppose $A\in\noninc$. Let $a\in\C$ and $r>0$ and $k,\ell\in\N$ with $k\neq \ell$. We say an abstract Swiss cheese $B=((b_n,s_n))$ is obtained from $A$ by {\em replacing} the disks $\ob{a_k}{r_k},\ob{a_\ell}{r_\ell}$ by $\ob ar$ if $B$ is obtained by deleting the disks at indices $k,\ell$ and inserting the disk $\ob ar$ at the first index in $\mathbb N$ such that the sequence $(s_n)_{n=1}^\infty$ is non-increasing.
\end{enumerate}
\end{definition}

Note that, if $A\in\noninc$, then the abstract Swiss cheese $B$ obtained by deleting or replacing disks, as defined in Definition~\ref{disk_mani}, is also in $\noninc$.

\section{Some geometric results}
Throughout, we shall require the following elementary geometric lemmas. The first is probably well-known, and the proof is elementary.

\begin{lemma} \label{lem1}
Let $z,w\in \C$ and $r,s\in \Rp,$ then $\cb{z}{r}\subseteq \cb{w}{s}$ if and only if $\abs{z-w}\leq s-r$. If $r>0$, then $\ob{z}{r}\subseteq \mathbb{C}\setminus \cb{w}{s}$ if and only if $\abs{w-z}\geq s+r$.
\end{lemma}

The following two elementary lemmas are essentially those used in \cite{feinheath2010,mason2010}, but including some additional information distilled from the original proofs. These lemmas are summarised in Figure \ref{elementary lemmas}. In the first lemma, we allow for the line segment to be degenerate.

\begin{figure}[htbp]
\centering
\begin{subfigure}[b]{0.47\textwidth}\centering
\begin{tikzpicture}
\draw[dashed] (-0.5,   0) circle [radius= 1.5];
\draw[dashed] ( 1.5, 0.5) circle [radius=1.25];
\draw[dashed] (0.379,0.22) circle [radius=2.41];
\node at (-1.12,-0.219) {$\ob{a_1}{r_1}$};
\node at (1.78,0.43) {$\ob{a_2}{r_2}$};
\node at (-0.136,1.89) {$\ob{a}{r}$};
\end{tikzpicture}
\caption{Combining open disks.}
\label{combininglem}
\end{subfigure}
~
\begin{subfigure}[b]{0.47\textwidth}
	\centering
\begin{tikzpicture}
\draw(-0.5,   0) circle [radius=   2];
\draw[dashed] ( 1.8,   0) circle [radius=1.35];
\draw (-1.03,   0) circle [radius=1.47];
\node at (-1.49,-0.0409) {$\cb{a}{r}$};
\node at (2.31,-0.00585) {$\ob{a_2}{r_2}$};
\node at (1.6,-1.73) {$\cb{a_1}{r_1}$};
\end{tikzpicture}
\caption{Pulling in the closed disk.}
\label{pullinginlem}
\end{subfigure}
\caption{Elementary lemmas for combining and pulling in disks.}
\label{elementary lemmas}
\end{figure}
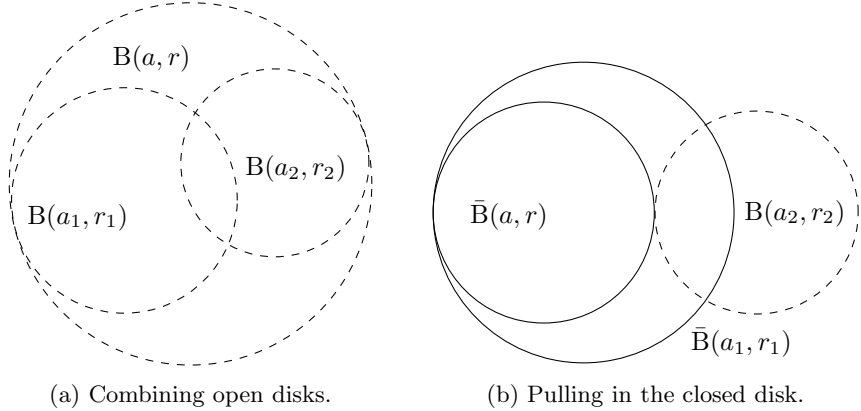

\begin{lemma}\label{combinedisks}
Let $a_1,a_2\in\C$ and $r_1,r_2>0$. Then there exists a unique pair $(a,r)\in\C\times\Rp$ with $\ob{a_1}{r_1}\cup\ob{a_2}{r_2}\subseteq\ob ar$ such that $r$ is minimal. Moreover, the point $a$ lies on the line segment joining $a_1$ and $a_2.$ Suppose further that $\cb{a_1}{r_1}\cap\cb{a_2}{r_2}\neq\emptyset$. Then $r\leq r_1+r_2,$ and equality holds if and only if $\ob{a_1}{r_1}\cap\ob{a_2}{r_2}=\emptyset$.
\end{lemma}

\begin{lemma}\label{outsidedisk}
Let $a_1,a_2\in\C$ and $r_1>r_2>0$ with $\cb{a_2}{r_2}\nsubseteq\ob{a_1}{r_1}$. Then there exists a unique pair $(a,r)\in\C\times\Rp$ with $\cb ar\subseteq\cb{a_1}{r_1}$ and $\ob{a_2}{r_2}\cap\cb ar=\emptyset$ such that $r$ is maximal. Moreover$,$ $r\geq r_1-r_2$ and equality holds if and only if $\ob{a_2}{r_2}\subseteq\ob{a_1}{r_1}$.
\end{lemma}

The cases in which equality holds in Lemmas \ref{combinedisks} and \ref{outsidedisk} are illustrated in Figure \ref{extreme_cases_fig}.

\begin{figure}[htbp]
\centering
\begin{subfigure}[b]{0.47\textwidth}
	\centering
\begin{tikzpicture}
\draw[dashed] (-1.25,   0) circle [radius= 1.5];
\draw[dashed] (1.25,   0) circle [radius=   1];
\draw[dashed] (-0.25,   0) circle [radius= 2.5];
\node at (-1.6,0.0804) {$\ob{a_1}{r_1}$};
\node at ( 1.2,0.0804) {$\ob{a_2}{r_2}$};
\node at (-0.242,1.79) {$\ob{a}{r}$};
\end{tikzpicture}
\caption{Case where equality holds in Lemma \ref{combinedisks}.}
\label{combining_extreme_case_fig}
\end{subfigure}
~
\begin{subfigure}[b]{0.47\textwidth}
	\centering
\begin{tikzpicture}
\draw(  -1,   0) circle [radius=2.35];
\draw[dashed] ( 0.5,   0) circle [radius=0.85];
\draw (-1.85,   0) circle [radius= 1.5];
\node at (-2,0) {$\cb{a}{r}$};
\node at (-0.487,1.57) {$\cb{a_1}{r_1}$};
\node at (0.538,0) {$\ob{a_2}{r_2}$};
\end{tikzpicture}
\caption{Case where equality holds in Lemma \ref{outsidedisk}.}
\label{smalldisk_extreme_case_fig}
\end{subfigure}
\caption{Extreme cases in the combining and pulling in lemmas.}
\label{extreme_cases_fig}
\end{figure}
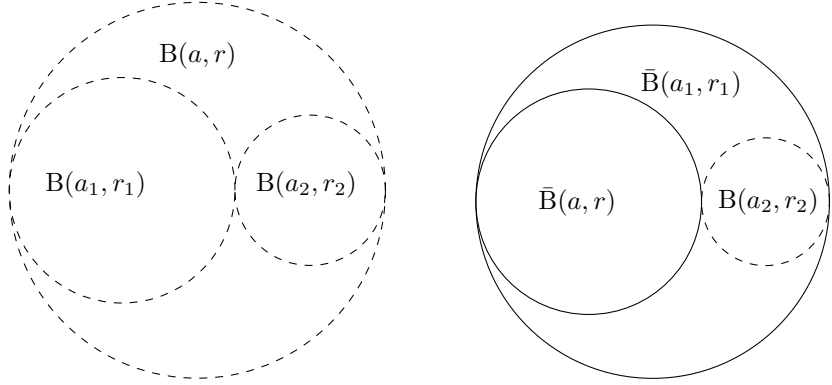

%%%%%%%%%%%%%%%%%%%%%%%%%%%%%%%%%%%%%%%
%%%%%%%%%%%%%%%%%%%%%%%%%%%%%%%%%%%%%%%
%%%%%%%%%%%%%%%%%%%%%%%%%%%%%%%%%%%%%%%
%%%%%%%%%%%%%%%%% Yang %%%%%%%%%%%%%%%%
\section{Classicalisation of Swiss cheeses}
\label{secclassicalisation}
We aim to give a topological proof of the Feinstein-Heath classicalisation theorem (Theorem \ref{feinheaththm}), as described in the introduction, stated below in the language of abstract Swiss cheeses.

\begin{theorem}\label{feinheaththm}
Let $A=((a_n,r_n))$ be an abstract Swiss cheese with $\delta_1(A)>0$. Then there exists a classical$,$ abstract Swiss cheese $B\in\scs$ such that $X_B\subseteq X_A$ and $\delta_1(B)\geq\delta_1(A)$.
\end{theorem}

We will see below that it is enough to prove this theorem for abstract Swiss cheeses where some redundancy has been eliminated, as the general case then follows. We first introduce the following terminology.

\begin{definition}
Let $A=((a_n,r_n))$ be an abstract Swiss cheese. Then $A$ is {\em redundancy-free} if$,$ for all $k\in S_A,$ we have $\ob{a_k}{r_k}\cap\cb{a_0}{r_0}\neq\emptyset$, and for all $\ell\in S_A$ with $k\neq \ell$ we have $\ob{a_k}{r_k}\not\subseteq\ob{a_\ell}{r_\ell}$.
\end{definition}

An elementary argument, which we leave to the reader, shows that it is easy to eliminate redundancy from abstract Swiss cheeses with finite radius sum, as in the following lemma.

\begin{lemma}\label{nonredundantcheese}
Let $A=((a_n,r_n))\in\scs$ with $\rho(A)<\infty.$ Then there exists a redundancy-free abstract Swiss cheese $B=((b_n,s_n))\in\noninc$ with $X_B=X_A,$ $\mu(B)<\infty$ and $\cb{b_0}{s_0}=\cb{a_0}{r_0}$ such that $\rho_E(B)\leq\rho_E(A)$ for each subset $E\subseteq\C$. In particular$,$ $\rho(B)\leq\rho(A)$.
\end{lemma}

Note that, since $\cb{b_0}{s_0}=\cb{a_0}{r_0}$ and $\rho(B)\leq\rho(A)$ in the above lemma we actually have $\delta_1(B)\geq\delta_1(A)$, as we claimed before. It is clear, by Lemma~\ref{nonredundantcheese}, that to prove Theorem~\ref{feinheaththm} it is enough to consider $A$ such that $\delta_1(A)>0$ and $A$ is redundancy-free.

We now define a relation on $\mathcal F$ which will help us to construct a compact subset of $\mathcal F$. Then we prove the existence of classical abstract Swiss cheeses with desired properties in this compact subset.

\begin{definition} \label{partially_above}
Let $A=((a_n,r_n))$ and $B=((b_n,s_n))$ be abstract Swiss cheeses. We say $B$ is \emph{partially above} $A$ if $\cb{b_0}{s_0}\subseteq \cb{a_0}{r_0},$ and, for each $n\in \N$, either $\ob{a_n}{r_n}\subseteq \C\setminus \cb{b_0}{s_0}$, or there exists $m\in \N$ such that $\ob{a_n}{r_n}\subseteq \ob{b_m}{s_m},$ or both.
\end{definition}

It is clear that $A$ is partially above itself and that if $B$ is partially above $A$, then $X_B\subseteq X_A$.

\bigskip

Fix a redundancy-free abstract Swiss cheese $A=((a_n,r_n))\in\noninc$ with $\delta_1(A)>0$. Note that $\rho(A)<\infty$ and, since $A$ is redundancy-free, $\mu(A)<\infty$. We set $R=\rho(A)$ and $M=\mu(A)$.

Let $\Sw$ be the collection of all $B=((b_n,s_n))\in\noninc(M,R)$ such that $B$ is partially above $A$.
Recall that, since $B\in\noninc(M,R)$, we have $s_n\leq R/n$ for all $n\in S_{B}$ so that
\begin{equation} \label{bound}
n\leq \frac R{r_n}\qquad(n\in S_B).
\end{equation}
By our conditions on $A$ it is clear that $A\in\Sw$. We now prove that $\Sw$ is compact.

\begin{lemma} \label{Closed_subset}
The set $\Sw$ is a compact subset of $\scs$.
\end{lemma}
\begin{proof}
{As noted earlier, it is enough to prove that $\Sw$ is closed in $\noninc(M,R)$ and that the $0$-th coordinate projection is bounded on $\Sw$. The latter is clear from the definition of $\Sw$, so we prove that $\Sw$ is closed in $\noninc(M,R)$.}

For each $m\in\No$, let $A^{(m)} = ((a_n^{(m)},r_n^{(m)}))_{n=0}^\infty$ be an abstract Swiss cheese in $\Sw$, and suppose the sequence $(A^{(m)})$ converges to $B=((b_n,s_n))\in\noninc(M,R)$. It remains to show that $B$ is partially above $A$.

It is easy to see (by Lemma~\ref{lem1}, for example) that $\cb{b_0}{s_0}\subseteq \cb{a_0}{r_0}$. Fix $k\in \mathbb N$. We show that either $\ob{a_k}{r_k}\subseteq \C\setminus\cb{b_0}{s_0}$ or there exists $\ell\in S_B$ with $\ob{a_k}{r_k}\subseteq\ob{b_\ell}{s_\ell}$. If $r_k=0$ then $\ob{a_k}{r_k}=\emptyset$ and the result is trivial, so we may assume that $k\in S_A$. First assume that there exists $n_0\in\No$ such that, for all $m\geq n_0$ we have
$\ob{a_k}{r_k} \subseteq \C\setminus \cb{a_0^{(m)}}{r_0^{(m)}}$. Then we have $\abs{a_k-a_0^{(m)}}\geq r_k+r_0^{(m)}$
for all $m\geq n_0$ by Lemma~\ref{lem1}. Letting $m\to \infty$, we obtain $\abs{a_k-a_0}\geq r_k+r_0$,
and so, by Lemma~\ref{lem1} again, $\ob{a_k}{r_k} \subseteq \C\setminus \cb{b_0}{s_0}.$

Otherwise for each $n_0\in\No$, there exist $m\geq n_0$ and $\ell_m\in\N$ such that
\begin{equation}\label{label1}
\ob{a_k}{r_k} \subseteq \ob{a_{\ell_m}^{(m)}}{r_{\ell_m}^{(m)}}.
\end{equation}
By passing to a subsequence of $A^{(m)}$ if necessary, we can assume \eqref{label1} holds for all $m\in\No$. For each $m$, since $r_{\ell_m}^{(m)}\geq r_k$, by \eqref{bound} we have $\ell_m\leq R/r_k$. Thus  there must be a $p\in \mathbb N$ that appears infinitely many times in the sequence $(\ell_m)_m$. Passing to a subsequence again if necessary, we may assume $\ell_m = p$ for all $m$. Since $A^{(m)}\to B$ as $m\to\infty$ and $\ob{a_k}{r_k} \subseteq \ob{a_{p}^{(m)}}{r_{p}^{(m)}}$,  it is again easy to show, using Lemma~\ref{lem1}, that $ \ob{a_k}{r_k} \subseteq \ob{b_p}{s_p}$.
Thus $B$ is partially above $A$ and we have proved that $\Sw$ is closed.
\end{proof}

Since $\delta_1$ is upper semicontinuous and $\Sw$ is compact and non-empty, $\delta_1$ attains a maximum value on $\Sw$ and this value is at least $\delta_1(A)>0$. Let
\[
\mathcal S_1 := \{ A'\in \Sw: \delta_1(A') = \sup_{B\in \Sw}\delta_1(B)\},
\]
which is also compact and non-empty.

\begin{lemma}\label{safe_replacement}
Let $B=((b_n,s_n))\in\Sw$.
\begin{enumerate}
  \item Suppose that $k,\ell\in S_B$ with $k\neq\ell$ such that $\cb{b_k}{s_k}\cap\cb{b_\ell}{s_\ell}\neq\emptyset$. If we have $\ob{b_k}{s_k}\cap\ob{b_\ell}{s_\ell}\neq\emptyset$ then there exists $B'\in\Sw$ such that $\delta_1(B')>\delta_1(B)$. Otherwise$,$ there exists $B'\in\Sw$ with $\delta_1(B')=\delta_1(B)$ and $\delta_2(B')<\delta_2(B)$.
  \item Suppose that $k\in S_B$ with $s_k<s_0$ such that $\cb{b_k}{s_k}\nsubseteq\ob{b_0}{s_0}$. If we have $\ob{b_k}{s_k}\nsubseteq\ob{b_0}{s_0}$ then there exists $B'\in\Sw$ such that $\delta_1(B')>\delta_1(B)$. Otherwise$,$ there exists $B'\in\Sw$ with $\delta_1(B')=\delta_1(B)$ and $\delta_2(B')<\delta_2(B)$.
\end{enumerate}
\end{lemma}
\begin{proof}
(a) Let $\ob bs$ be the open disk obtained by applying Lemma \ref{combinedisks} to the disks $\ob{b_k}{s_k}$ and $\ob{b_\ell}{s_\ell}$. Let $B'=((b_n',s_n'))$ be obtained by replacing the disks $\ob{b_k}{s_k}$ and $\ob{b_\ell}{s_\ell}$ by $\ob bs$.

If $\ob{b_k}{s_k}\cap\ob{b_\ell}{s_\ell}\neq\emptyset$ then we have $s<s_k+s_\ell$ and so $\delta_1(B')>\delta_1(B)$. Otherwise, we have $s=s_k+s_\ell$ and hence $s^2>s_k^2+s_\ell^2$. In this case, we have $\delta_1(B')=\delta_1(B)$ and $\delta_2(B')<\delta_2(B)$.

We now show that $B'\in\Sw$. Clearly $B'\in\noninc$ by our definition of replacing disks in an abstract Swiss cheese. Since $b$ lies on the line segment connecting $b_k$ and $b_\ell$, it follows that $\mu(B')\leq\mu(B)$ and since $s\leq s_k+s_\ell$ we have $\rho(B')\leq\rho(B)$. Thus $B'\in\noninc(M,R)$. It remains to show that $B'$ is partially above $A$.

We have $\cb{b_0'}{s_0'}=\cb{b_0}{s_0}$ so that $\cb{b_0'}{s_0'}\subseteq\ob{a_0}{r_0}$. Fix $p\in\N$. Since $B$ is partially above $A$, we have $\ob{a_p}{r_p}\subseteq\ob{b_m}{s_m}$ for some $m\in S_B$ or $\ob{a_p}{r_p}\subseteq\C\setminus\cb{b_0}{s_0}$. If $\ob{a_p}{r_p}\subseteq\C\setminus\cb{b_0}{s_0}$ then we also have $\ob{a_p}{r_p}\subseteq\C\setminus\cb{b_0'}{s_0'}$. Otherwise, let $m\in S_B$ with $\ob{a_p}{r_p}\subseteq\ob{b_m}{s_m}$. If $m=k$ or $m=\ell$, then, with $q$ as the index where $\ob bs$ was inserted, we have $\ob{a_p}{r_p}\subseteq\ob{b_q'}{s_q'}$. If $m\neq k,\ell$, then there exists $q\in S_{B'}$ such that $\ob{b_q'}{s_q'}=\ob{b_m}{s_m}$. Thus $\ob{a_p}{r_p}\subseteq\ob{b_q'}{s_q'}.$ Hence $B'$ is partially above $A$, and so $B'\in\Sw$ as required.

(b) Let $\cb bs$ be the closed disk obtained by applying Lemma \ref{outsidedisk} to the disks $\ob{b_0}{s_0}$ and $\ob{b_k}{s_k}$. Let $B'=((b_n',s_n'))$ be the abstract Swiss cheese obtained by deleting the disks at indices $0$ and $k$ and inserting the disk $\cb bs$ at index $0$.

If $\ob{b_k}{s_k}\nsubseteq\ob{b_0}{s_0}$ then we have $s>s_0-s_k$ so that $\delta_1(B')>\delta_1(B)$. Otherwise, we have $s_0=s+s_k$ and $s_0^2>s^2+s_k^2$ so that $\delta_1(B')=\delta_1(B)$ and $\delta_2(B')<\delta_2(B)$.

The proof that $B'\in\Sw$ is similar to the proof in part (a).
%
%
%It is easy to see that $B'\in\noninc(M,R)$ and we have $\ob bs\subseteq \ob{b_0}{s_0}$ by construction.
%
%
%
%Fix $p\in\N$. Again, we show that $\ob{a_p}{r_p}\subseteq(\C\setminus\cb bs)$ or that there exists $q\in S_{B'}$ such that $\ob{a_p}{r_p}\subseteq\ob{b_q'}{s_q'}$. If $r_p=0$ then the result is trivial, as above, so suppose $p\in S_A$. Since $B\in\Sw$, we have $\ob{a_p}{r_p}\subseteq(\C\setminus\cb{b_0}{s_0})$ or we have $\ob{a_p}{r_p}\subseteq \ob{b_m}{s_m}$ for some $m\in S_B$. If $\ob{a_p}{r_p}\subseteq(\C\setminus\cb{b_0}{s_0})$ then we also have $\ob{a_p}{r_p}\subseteq(\C\setminus\cb bs)$. Let $m\in S_B$ and suppose that $\ob{a_p}{r_p}\subseteq\ob{b_m}{s_m}$. If $m\neq k$ then there exists $q\in S_{B'}$ such that $\ob{b_q'}{s_q'}=\ob{b_m}{s_m}$ and hence $\ob{a_p}{r_p}\subseteq\ob{b_q'}{s_q'}$. If $m=k$ then we now have $\ob{a_p}{r_p}\subseteq(\C\setminus\cb bs)$. It follows that $B'\in\Sw$.
\end{proof}

We are now ready to prove the main results of this section.

\begin{theorem} \label{semi_classicalization}
All abstract Swiss cheeses in $\mathcal S_1$ are semiclassical.
\end{theorem}
\begin{proof}
Let $B=((b_n,s_n))\in \mathcal S_1$. Suppose for contradiction that $B$ is not a semiclassical abstract Swiss cheese. Consider first the case where there are distinct $k,\ell\in S_B$ with $\ob{b_k}{s_k}\cap \ob{b_\ell}{s_\ell} \neq \emptyset$. By Lemma \ref{safe_replacement}(a) there exists $B'\in\Sw$ with $\delta_1(B')>\delta_1(B)$, which is a contradiction.

The remaining case is where there is a $k\in S_B$ with $\ob{b_k}{s_k}\nsubseteq \ob{b_0}{s_0}$. We have $\delta_1(B)\geq\delta_1(A)>0$ so that $s_k<s_0$. By Lemma \ref{safe_replacement}(b) there exists $B'\in\Sw$ with $\delta_1(B')>\delta_1(B)$, which is a contradiction.
\end{proof}

Since $\mathcal S_1$ is compact and non-empty, $\delta_2$ attains both maximum and minimum values on $\mathcal S_1$. Let
\[
\mathcal S_2 := \{ A'\in \mathcal S_1: \delta_2(A') = \inf_{B\in \mathcal S_1} \delta_2(B) \},
\]
which is again non-empty and compact. Since all the abstract Swiss cheeses in $\mathcal S_1$ are semiclassical, $\pi\delta_2(B)$ is the area of $X_B$ for all $B\in \mathcal S_1$, and hence for all $B\in\mathcal S_2$. So the abstract Swiss cheeses in $\mathcal S_2$ are obtained by finding those $B\in\mathcal S_1$ for which the area of $X_B$ is minimal on $\mathcal S_1$.

\begin{theorem} \label{classicalization}
All abstract Swiss cheeses in $\mathcal S_2$ are classical.
\end{theorem}
\begin{proof}
Let $B=((b_n,s_n))\in \mathcal S_2$. Suppose for contradiction that $B$ is not classical. If there are distinct $k,\ell\in S_B$ with $\cb{b_k}{s_k}\cap \cb{b_\ell}{s_\ell}\neq \emptyset$ then, by Lemma \ref{safe_replacement}(a), there exists $B'\in\Sw$ such that either $\delta_1(B')>\delta_1(B)$ or $\delta_1(B')=\delta_1(B)$ and $\delta_2(B')<\delta_2(B)$. In either case we obtain a contradiction since $B\in\mathcal S_2$.

Otherwise there exists $k\in S_B$ with $\cb{b_k}{s_k}\nsubseteq \ob{b_0}{s_0}$. Note that $s_k<s_0$ since $\delta_1(B)>0$. By Lemma \ref{safe_replacement}(b) there exists $B'\in\Sw$ such that either $\delta_1(B')>\delta_1(B)$ or $\delta_1(B')=\delta_1(B)$ and $\delta_2(B')<\delta_2(B)$. In either case we obtain a contradiction since $B\in\mathcal S_2$.
\end{proof}

In the next theorem, we show that if $X_A$ has empty interior then we do not have to minimise $\delta_2$ on $\mathcal S_1$ to find classical abstract Swiss cheeses.

\begin{theorem} \label{non_interior_classicalization}
If $\tint{X_{A}}=\emptyset$ then each abstract Swiss cheese in $\mathcal S_1$ is classical.
\end{theorem}

\begin{proof}
Let $B=((b_n,s_n))\in \mathcal S_1$. Then, by Theorem~\ref{semi_classicalization}, $B$ is semiclassical. Suppose for contradiction that $B$ is not classical. Then there are two cases summarised in Figure \ref{fig:No_Interior_Proof}. First suppose there exist distinct $k,\ell\in S_B$ with $\cb{b_k}{s_k}\cap \cb{b_\ell}{s_\ell}\neq \emptyset$.  Then by Lemma~\ref{combinedisks}, since $\ob{b_k}{s_k}\cap\ob{b_\ell}{s_\ell}=\emptyset$, there exists an open disk $\ob{a}{r}\supseteq \ob{b_k}{s_k}\cup \ob{b_\ell}{s_\ell}$ with $r=s_k+s_\ell$. By replacing the disks $\ob{b_k}{s_k}$ and $\ob{b_\ell}{s_\ell}$ with $\ob ar$ we obtain a new abstract Swiss cheese $B'=((b_n',s_n'))$ such that $B'\in \mathcal S_1$ (following the proof of Lemma \ref{safe_replacement}). Let $p$ be the index at which the disk $\ob ar$ was inserted. Since $X_{B}$ has empty interior, there exists $m\in S_{B}$ with $m\neq p$ such that $\ob{a}{r}\cap \ob{b_m}{s_m}\neq \emptyset$. Let $q\in S_{B'}$ be such that $\ob{b_q'}{s_q'}=\ob{b_m}{s_m}$. Note that $p\neq q$. Applying Lemma \ref{safe_replacement}(a) to $p,q\in S_{B'}$ and $B'$, we obtain an abstract Swiss cheese $B''\in\Sw$ which has $\delta_1(B'')>\delta_1(B')$. But this is a contradiction.

Now suppose there exists $k\in S_B$ with $\cb{b_k}{s_k}\nsubseteq \ob{b_0}{s_0}$.
Let $\cb bs$ be the closed disk obtained by applying Lemma \ref{outsidedisk} to the disks $\cb{b_0}{s_0}$ and $\ob{b_k}{s_k}$. Since $B$ is semiclassical, we have $s=s_0-s_k$ (as in Figure \ref{smalldisk_extreme_case_fig}). By deleting the disks at indices $0$ and $k$ and inserting $\ob bs$ at index $0$, we obtain a new abstract Swiss cheese $B'=((b_n',s_n'))\in \mathcal S_1$ such that $\delta_1(B')=\delta_1(B)$ (again following the proof of Lemma \ref{safe_replacement}). Since $X_{B}$ has empty interior, there exists $q\in S_{B'}$ such that $\ob{b_q}{s_q}\nsubseteq \cb{b}{s}$. Applying Lemma \ref{safe_replacement}(b) to $q$ and $B'$, we obtain an abstract Swiss cheese $B''\in\Sw$ which has $\delta_1(B'')>\delta_1(B')$. But this is a contradiction.
\end{proof}

Let $B$ be an abstract Swiss cheese satisfying $\delta_1(B)>0$, so that $B$ satisfies the conditions of Theorem~\ref{feinheaththm}. Then we can apply Lemma \ref{nonredundantcheese} to obtain a redundancy-free abstract Swiss cheese $A\in\noninc$ with $X_A=X_B$ and such that $\delta_1(A)\geq\delta_1(B)$. We can then apply the above constructions to $A$. Each abstract Swiss cheese $A'$ from the corresponding non-empty set $\mathcal S_2$ is classical by Theorem \ref{classicalization} and has $X_{A'}\subseteq X_{A}=X_B$ and $\delta_1(A')\geq\delta_1(A)\geq\delta_1(B)$. So we obtain the Feinstein-Heath classicalisation theorem as a corollary of Theorem \ref{classicalization}.

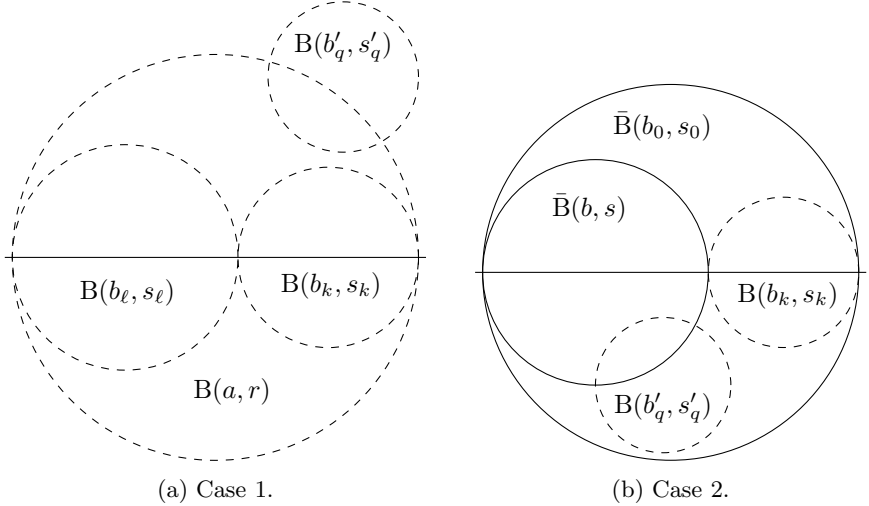
\begin{figure}[htbp]
	\centering
\begin{subfigure}[b]{0.47\textwidth}\centering
    \begin{tikzpicture}
    \draw [dashed] (-1.5,  0) circle [radius= 1.5];
    \draw [dashed] ( 1.2,  0) circle [radius= 1.2];
    \draw [dashed] (-0.3,  0) circle [radius= 2.7];
    \draw [dashed] (1.4, 2.4) circle [radius= 1.0];
    \node at (-1.47,-0.465) {$\ob{b_\ell}{s_\ell}$};
    \node at (1.24,-0.377) {$\ob{b_k}{s_k}$};
    \node at (-0.076,-1.78) {$\ob ar$};
    \node at (1.4,2.8) {$\ob{b'_q}{s'_q}$};
    \draw (-3.1,0) -- (2.5,0);
    \end{tikzpicture}

	\caption{Case 1.}
	\label{fig:expansion}
\end{subfigure}
~
\begin{subfigure}[b]{0.47\textwidth}
	\centering
    \begin{tikzpicture}
    \draw (-1.5,   0) circle [radius= 1.5];
    \draw [dashed] (   1,   0) circle [radius=   1];
    \draw (-0.5,   0) circle [radius= 2.5];
    \draw [dashed] (-0.6, -1.5) circle [radius=0.9];
    \node at (-0.609,1.94) {$\cb{b_0}{s_0}$};
    \node at (1.06,-0.3) {$\ob{b_k}{s_k}$};
    \node at (-1.58,0.855) {$\cb{b}{s}$};
    \node at (-0.6,-1.8) {$\ob{b'_q}{s'_q}$};
    \draw (-3.1,0) -- (2.1,0);
    \end{tikzpicture}

	\caption{Case 2.}
	\label{fig:pulling_in}
\end{subfigure}
\caption{The two cases in the proof of Theorem \ref{non_interior_classicalization}}
\label{fig:No_Interior_Proof}
\end{figure}

%%%%%%%%%%%%%%%%%%%%%%%%%%%%%%%%%%%%%%%
%%%%%%%%%%%%%%%%%%%%%%%%%%%%%%%%%%%%%%%
%%%%%%%%%%%%%%%%%%%%%%%%%%%%%%%%%%%%%%%
%%%%%%%%%%%%%%%%%%%%%%%%%%%%%%%%%%%%%%%

\section{Controlled classicalisation}
\label{seclocal}
In this section we discuss some situations in which it is possible to make a Swiss cheese classical without changing certain disks. This process we call ``controlled classicalisation''.

Recall that, for $E\subseteq\C$ and an abstract Swiss cheese $A=((a_n,r_n))$, the set $H_A(E)$ is the set of all $n\in S_A$ such that $\cb{a_n}{r_n}\cap E\neq\emptyset$.

\begin{lemma}\label{diskintsums}
Let $U$ be a non-empty open subset of $\C$. For each $m\in\No,$ let $A^{(m)}=((a_n^{(m)},r_n^{(m)}))\in\scs$ and suppose that $A^{(m)}\to A=((a_n,r_n))\in \mathcal F$ as $m\to\infty$. Then $\rho_U(A)\leq \liminf_{m\to\infty}\rho_U(A^{(m)})$.
\end{lemma}
\begin{proof}
Since $U$ is open and $A^{(m)}\to A$ as $m\to\infty$, for each $k\in H_A(U)$ there exists $m_0\in\No$ such that, for all $m\geq m_0$, we have $k\in S_{A^{(m)}}$ and
\[
\cb{a_k^{(m)}}{r_k^{(m)}}\cap U\neq\emptyset\,.
 \]
Let $\chi_m$ denote the characteristic function of $H_{A^{(m)}}(U)\cap H_A(U)$. Then $\chi_m$ converges pointwise to $\chi:=\chi_{H_A(U)}$ as $m\to\infty$. Since $r_k^{(m)}\to r_k$ as $m\to\infty$ for each $k$, by Fatou's lemma for series, we have
\[
\rho_U(A)=\sum_{n=1}^\infty\chi(n)r_n\leq\liminf_{m\to\infty}\sum_{n=1}^\infty\chi_m(n)r_n^{(m)}\leq\liminf_{m\to\infty}\rho_U(A^{(m)}),
\]
as required.
\end{proof}

For the rest of this section $A=((a_n,r_n))\in\noninc$ will be a fixed redundancy-free abstract Swiss cheese. Note that both $\rho(A)$ and $\mu(A)$ are finite and $r_n\leq \rho(A)/n$ for all $n\in\N$. We define the (classical) {\em error set of $A$} to be
\[
E(A):=\bigcup\limits_{\substack{m,n\in S_A\\m\neq n}}\bigg(\cb{a_m}{r_m}\cap \cb{a_n}{r_n}\bigg)\cup\bigcup\limits_{n\in S_A}((\C\setminus\ob{a_0}{r_0})\cap \cb{a_n}{r_n}).
\]
Note that if $E(A)\subseteq\ob{a_0}{r_0}$ then $\cb{a_n}{r_n}\subseteq\ob{a_0}{r_0}$ for all $n\in S_A$. We aim to prove that, under suitable conditions, we can classicalise $A$ while leaving many of the open disks unchanged.

As in Section \ref{secclassicalisation}, we seek to construct a compact subset of $\scs$ on which the function $\delta_1$ can be maximised and then the function $\delta_2$ minimised to give a suitable classical abstract Swiss cheese.

In the rest of this paper, we will frequently need to consider indexed collections of pairs of sets of the following form. Let $I\subseteq\N$ be non-empty. Let $\mathcal C=((K_n,U_n))_{n\in I}$, where each $K_n$ is a compact plane set and each $U_n$ is an open set with $K_n\subseteq U_n$. We call such an indexed collection a {\em controlling collection of pairs}. In the special case where $I$ has only one member, we say $\mathcal C$ is a {\em controlling pair} and write $\mathcal C=(K,U)$.

\begin{definition} \label{locally_partially_above}
Let $\mathcal C=((K_n,U_n))_{n\in I}$ be a controlling collection of pairs. Define
\[
V(\mathcal C):=\bigcup\limits_{n\in I}U_n,\quad F(\mathcal C):=\bigcup\limits_{n\in I}K_n.
\]
Let $\La C$ denote the set of all $B=((b_n,s_n))\in\noninc(\mu(A),\rho(A))$ such that$:$
\begin{enumerate}
 \item for each $(K,U)\in\mathcal C$ we have $\rho_U(B)\leq\rho_U(A);$
 \item $\cb{b_0}{s_0}=\cb{a_0}{r_0};$
 \item for all $k\in S_A$ with $\cb{a_k}{r_k}\cap V(\mathcal C)=\emptyset$ there exists $\ell\in S_B$ such that $\ob{b_\ell}{s_\ell}=\ob{a_k}{r_k};$
 \item for each $n\in I$ and for all $k\in S_A$ with $\cb{a_k}{r_k}\cap U_n\neq\emptyset:$
 \begin{enumerate}
  \item there exists $\ell\in S_B$ with $\ob{b_\ell}{s_\ell}=\ob{a_k}{r_k};$ or
  \item there exists $\ell\in H_B(K_n)$ with $\ob{a_k}{r_k}\subseteq\ob{b_\ell}{s_\ell}$.
 \end{enumerate}
\end{enumerate}
\end{definition}

Note that $A\in\La C$, and if $B\in\La C$ then $B$ is partially above $A$. Thus if $B\in\La C$ then $X_B\subseteq X_A$. The properties (a)-(d) reflect the properties we desire for the final abstract Swiss cheese. We will use the open sets $U$ to bound the error set $E(A)$. Under some technical assumptions, conditions (c) and (d) ensure that abstract Swiss cheeses maximising $\delta_1$ in $\La C$ have the property that any open disk which lies outside $V(\mathcal C)$ is the same as an open disk from $A$.

We first require some preliminary lemmas. The following lemma is probably well-known and can be proved using a Hausdorff metric argument, but we include an elementary proof for the convenience of the reader.

\begin{lemma}\label{limitdiskintersection}
Let $K$ be a compact plane set. Let $(z_n)$ be a sequence in $\C,$ and let $(t_n)$ be a sequence in $\Rp$. Suppose that $\cb{z_n}{t_n}\cap K\neq\emptyset$ for all $n$, and that $z_n\to z$ and $t_n\to t$ as $n\to\infty$. Then $\cb zt\cap K\neq\emptyset$.
\end{lemma}
\begin{proof}For each $n\in\No$ there exists a point $w_n\in\cb{z_n}{t_n}\cap K$. Now since $(w_n)$ is a sequence in $K$ there is a convergent subsequence $(w_{n_k})$ converging to a point $w\in K$. For each $k\in\No,$ we have $w_{n_k}\in\cb{z_{n_k}}{t_{n_k}}$ so that $\abs{w_{n_k}-z_{n_k}}\leq t_{n_k}$. Hence, taking the limit as $k\to\infty$, we have $\abs{w-z}\leq t$ so that $w\in\cb{z}{t}\cap K$ as required.
\end{proof}

We now prove that the space $\La C$ is a compact subspace of $\scs$ for an arbitrary countable collection $\mathcal C$ of pairs $(K,U)$ where $K$ is a compact plane set and $U$ an open neighbourhood of $K$.

\begin{lemma}\label{Lcompactness}
Let $\mathcal C:=((K_n,U_n))_{n\in I}$ be a controlling collection of pairs. Then the set $\La C\subseteq\scs$ is compact.
\end{lemma}
\begin{proof}
{It is sufficient to show that $\La C$ is closed in $\noninc(\mu(A),\rho(A))$, since the $0$-th coordinate projection is clearly bounded on $\La C$}. For each $m\in\No$, let $A^{(m)}=((a_n^{(m)},r_n^{(m)}))_{n=0}^\infty\in\La C$. Let $B=((b_n,s_n))$ and suppose that $A^{(m)}\to B\in\noninc(\mu(A),\rho(A))$ as $m\to\infty$; we need to show that $B\in\La C$.

By Lemma \ref{diskintsums} we see that $B$ also satisfies (a), and it is immediate that (b) is also satisfied.

It remains to prove (c) and (d) hold for $B$. Fix $k\in S_A$. Suppose that $\cb{a_k}{r_k}\cap V(\mathcal C)=\emptyset$. Since, for each $m\in\No$, we have $A^{(m)}\in\La C$ it follows that for each $m$ there exists an integer $\ell_m$ such that $\ob{a_k}{r_k}=\ob{a_{\ell_m}^{(m)}}{r_{\ell_m}^{(m)}}$. Now since $r_k^{(m)}=r_k$ for each $m$ we have $1\leq\ell_m\leq \rho(A)/r_k$ for all $m$. But then there must exist an integer $1\leq p\leq \rho(A)/r_k$ such that $\ell_k=p$ infinitely often so we can find a subsequence $(A^{(m_j)})_j$ such that $\ell_{m_j}=p$ for all $j$. Since $\ob{a_k}{r_k}=\ob{a_k^{(m_j)}}{r_k^{(m_j)}}$ for all $j$ and $A^{(m_j)}\to B$ as $j \to \infty$, it follows that $\ob{a_k}{r_k}=\ob{b_p}{s_p}$. This proves that (c) holds for $B$.

Now suppose that $\cb{a_k}{r_k}\cap U\neq\emptyset$ for some $(K,U)\in\mathcal C$. As above, for each $m\in\No$ there exists an integer $\ell_m$ such that $\ob{a_k}{r_k}\subseteq\ob{a_{\ell_m}^{(m)}}{r_{\ell_m}^{(m)}}$ and $r_{\ell_m}^{(m)}\geq r_k$. We choose $\ell_m$ as follows: if in $A^{(m)}$ there is an open disk $\ob{a}{r} = \ob{a_k}{r_k}$ then we pick $\ell_m$ to be the index of that open disk, otherwise we choose $\ell_m$ to be the index of an open disk $\ob{a}{r}$ that properly contains $\ob{a_k}{r_k}$ and $\cb{a}{r}\cap F(\mathcal C)\neq \emptyset$.
Hence we have $1\leq\ell_m\leq \rho(A)/r_k$ for all $m$ and so there exists an integer $1\leq p\leq \rho(A)/r_k$ such that $\ell_m=p$ infinitely often. By considering a subsequence we can assume that $\ell_m=p$ for all $m$.
If $\ob{a_{p}^{(m)}}{r_{p}^{(m)}}=\ob{a_k}{r_k}$ holds for infinitely many $m$ then there is a subsequence $(A^{(m_j)})_j$ such that $\ob{a_k}{r_k}=\ob{a_p^{(m_j)}}{r_p^{(m_j)}}$ for all $j$. Since $A^{(m_j)}\to B$ as $j\to\infty$ it follows that $\ob{a_k}{r_k}=\ob{b_p}{s_p}$.
If $\ob{a_k}{r_k}=\ob{a_p^{(m)}}{r_p^{(m)}}$ for only finitely many $m$ then we must have \[
\ob{a_k}{r_k}\subseteq\ob{a_p^{(m)}}{r_p^{(m)}}\quad\text{and}\quad \cb{a_p^{(m)}}{r_p^{(m)}}\cap K\neq\emptyset
\]
for infinitely many $m$. Then there exists a subsequence $(A^{(m_j)})_j$ such that $\ob{a_k}{r_k}\subseteq\ob{a_p^{(m_j)}}{r_p^{(m_j)}}$ and $\cb{a_p^{(m_j)}}{r_p^{(m_j)}}\cap K\neq\emptyset$ for all $j$. But then $\ob{a_k}{r_k}\subseteq\ob{b_p}{s_p}$ and, by Lemma \ref{limitdiskintersection}, we have $\cb{b_p}{s_p}\cap K\neq\emptyset$. This proves that (d) holds for $B$.

Thus we have proved that $B\in\La C$ and hence $\La C$ is compact.
\end{proof}

We are interested in those abstract Swiss cheeses $B$ in a space $\La C$ on which the discrepancy function $\delta_1$ is maximised. These abstract Swiss cheeses have some desirable properties. Let $\LM{C}$ denote the subset of $\La C$ of all abstract Swiss cheeses where $\delta_1$ achieves its maximum. Since $\La C$ is non-empty and compact, $\LM{C}$ is non-empty and compact. Recall that $A\in\noninc$ is assumed to be redundancy-free.

\begin{lemma}\label{diskscorrespondence}
Let $\mathcal C:=((K_n,U_n))_{n\in I}$ be a controlling collection of pairs. Let $B=((b_n,s_n))\in\LM C$. Then $B$ has the following properties.
\begin{enumerate}
  \item For all $k,\ell\in S_B$ with $k\neq\ell$, we have $\ob{b_k}{s_k}\neq\ob{b_\ell}{s_\ell}$.
  \item For each $k\in S_B$, there exists $\ell\in S_A$ such that $\ob{a_\ell}{r_\ell}\subseteq\ob{b_k}{s_k}$. Moreover$,$ if $\cb{b_k}{s_k}\cap F(\mathcal C)=\emptyset$ then this $\ell\in S_A$ is unique$,$ and we have $\ob{b_k}{s_k}=\ob{a_\ell}{r_\ell}$.
  \item Let $E$ be a fixed subset of $\C$. Let $H_1:=H_B(E)\setminus H_B(V(\mathcal C))$ and let $H_2:=H_A(E)\setminus H_A(V(\mathcal C))$. There exists a bijection $\sigma:H_1\to H_2$ satisfying the following condition$:$ for each $k\in H_1$ and $\ell\in H_2,$ we have $\sigma(k)=\ell$ if and only if $\ob{b_k}{s_k}=\ob{a_\ell}{s_\ell}$. In particular$,$
      \[
      \sum_{n\in H_1}s_n=\sum_{n\in H_2}r_n.
      \]
\end{enumerate}
\end{lemma}
\begin{proof}
(a) If $k,\ell\in S_B$ with $k\neq\ell$ such that $\ob{b_k}{s_k}=\ob{b_\ell}{s_\ell}$ then we can obtain an abstract Swiss cheese $B'$ by deleting the disk at index $\ell$ which has $\delta_1(B')>\delta_1(B)$. It is easy to see that $B'\in\La C$, which is a contradiction.

(b) Let $k\in S_B$. Assume, for contradiction, there does not exist $\ell\in S_A$ such that $\ob{a_\ell}{r_\ell}\subseteq\ob{b_k}{s_k}$. Then we can delete the disk at index $k$ from $B$ to obtain an abstract Swiss cheese $B'$ which has $\delta_1(B')>\delta_1(B)$. It is clear that $B'\in\La C$, which contradicts the maximality of $\delta_1(B)$. Thus there exists $\ell\in S_A$ such that $\ob{a_\ell}{r_\ell}\subseteq\ob{b_k}{s_k}$.

Now suppose, in addition, that $\cb{b_k}{s_k}\cap F(\mathcal C)=\emptyset$. We show that the $\ell\in S_A$ found above with $\ob{a_\ell}{r_\ell}\subseteq\ob{b_k}{s_k}$ is unique and that we have $\ob{a_\ell}{r_\ell}=\ob{b_k}{s_k}$. Assume, for contradiction, that $\ob{a_\ell}{r_\ell}\neq\ob{b_k}{s_k}$. Then, since $A$ is redundancy-free, we must have $\ob{a_m}{r_m}\neq\ob{b_k}{s_k}$ for all $m\in S_A$. We claim that the abstract Swiss cheese $B'$ obtained by deleting the disk at index $k$ from $B$ has $B'\in\La C$; this will lead to a contradiction.

Clearly $B'\in\noninc(\mu(A),\rho(A))$ and it is also clear that $B'$ satisfies conditions (a) and (b) of Definition \ref{locally_partially_above}(a). Since $\ob{a_m}{r_m}\neq\ob{b_k}{s_k}$ for all $m\in S_A$, it follows that \ref{locally_partially_above}(c) remains true for $B'$. Similarly, since $\cb{b_k}{s_k}\cap F(\mathcal C)=\emptyset$, \ref{locally_partially_above}(d) remains true for $B'$. This proves our claim.

But now $\delta_1(B')>\delta_1(B)$, which contradicts the maximality of $\delta_1(B)$. Thus we must have $\ob{a_\ell}{r_\ell}=\ob{b_k}{s_k}$. The uniqueness of $\ell$ follows from the fact that $A$ is redundancy-free.

(c) Note that if, for some $k\in S_B$ and $\ell\in S_A$, $\ob{b_k}{s_k}=\ob{a_\ell}{r_\ell}$ then $k\in H_1$ if and only if $\ell\in H_2$. Combining this with (b), for each $k\in H_1$ there exists a unique $\ell\in H_2$ such that $\ob{b_k}{s_k}=\ob{a_\ell}{r_\ell}$. Thus we may define $\sigma(k)=\ell$ for such $k,\ell$.
We must show that $\sigma$ is a bijection. By (a), $\sigma$ is injective. Let $\ell\in H_2$. By \ref{locally_partially_above}(c), there exists $k\in S_B$ with $\ob{b_k}{s_k}=\ob{a_\ell}{r_\ell}$. By the remark above, $k\in H_1$, and so $\sigma(k)=\ell$. This proves that $\sigma$ is surjective. It is now immediate that $\sum_{n\in H_1}s_n=\sum_{n\in H_2}r_n$. This completes the proof.
\end{proof}

In order to obtain a controlled classicalisation theorem, we need to impose some technical conditions on $\mathcal C$. Recall that if $E\subseteq \C$ is non-empty  and $z\in\C$ then we define the distance of $z$ to $E$ by $\dist(z,E):=\inf\{\abs{z-x}:x\in E\}$. For a non-empty compact set $K\subseteq\C$ and positive real number $M$ we define $U(K,M):=\{z\in\C:\dist(z,K)< M\}$.

\begin{lemma}\label{newinLforseq}
Let $I\subseteq\N$ be non-empty. Let $(K_n)_{n\in I}$ be a collection of compact plane sets and let $(M_n)_{n\in I}$ be a collection of positive real numbers. For each $n\in I,$ let $U_n:=U(K_n,M_n)$. Suppose that $\rho_{U_k}(A)<M_k/2$ and $U_k\subseteq\ob{a_0}{r_0}$ for all $k\in I$ and suppose that $U_k\cap U_\ell=\emptyset$ for all distinct $k,\ell\in I$.
Let $\mathcal C$ be the controlling collection $((K_n,U_n))_{n\in I}$. Let $B=((b_n,s_n))\in\La C$ and fix $m\in I$.
Suppose there exists $k,\ell\in S_B$ with $k\neq \ell$ such that $\cb{b_k}{s_k}\cap K_m\neq\emptyset$ and $\cb{b_k}{s_k}\cap\cb{b_\ell}{s_\ell}\neq\emptyset$. Then there exists $B'\in\La C$ such that either $\delta_1(B')>\delta_1(B)$ or $\delta_1(B')=\delta_1(B)$ and $\delta_2(B')<\delta_2(B)$.
\end{lemma}
\begin{proof}
Let $\ob bs$ be the disk obtained by the application of Lemma \ref{combinedisks} to the disks $\ob{b_k}{s_k}$ and $\ob{b_\ell}{s_\ell}$. Let $B'=((b_n',s_n'))$ be an abstract Swiss cheese obtained from $B$ by replacing the disks at indices $k,\ell$ with the disk $\ob bs$. Since $B\in\La C$ we have $\rho_{U_m}(B)\leq\rho_{U_m}(A)<M_m/2$, so that $s\leq s_k+s_\ell<M_m/2$. Since $\cb{b_k}{s_k}\cap K_m\neq\emptyset$, we must have $\cb bs\subseteq U_m$ and hence $\cb bs\cap U_n=\emptyset$ for all $n\in I$ with $n\neq m$.

It is clear now that either $\delta_1(B')>\delta_1(B)$, when $s<s_k+s_\ell$, or we have $\delta_1(B')=\delta_1(B)$ and $\delta_2(B')<\delta_2(B)$, when $s=s_k+s_\ell$, so it remains to show that $B'\in\La C$. By construction, and since $\cb bs\subseteq U_m$ and $\cb bs\cap U_n=\emptyset$ for $n\in I$ with $n\neq m$, we have $B'\in\noninc(\mu(A),\rho(A))$ and satisfies (a) and (b) in Definition~\ref{locally_partially_above}.

Fix $j\in S_A$. If $\cb{a_j}{r_j}\cap V(\mathcal C)=\emptyset$, then there exists $p\in S_B$ with $p\neq k,\ell$ and $\ob{b_p}{s_p}=\ob{a_j}{r_j}$. Hence there is a $p'\in S_{B'}$ such that $\ob{b_{p'}'}{s_{p'}'}=\ob{a_j}{r_j}$ and $B'$ satisfies (c) in Definition~\ref{locally_partially_above}.

Suppose that $\cb{a_j}{r_j}\cap V(\mathcal C)\neq\emptyset$. Let $n\in I$ such that $\cb{a_j}{r_j}\cap U_n\neq\emptyset$. Since $B\in\La C$, there exists $p\in S_B$ such that $\ob{a_j}{r_j}\subseteq\ob{b_p}{s_p}$, where equality holds unless $\cb{b_p}{s_p}\cap K_n\neq\emptyset$. If $p\neq k,\ell$, then there exists $q\in S_{B'}$ such that $\ob{b_q'}{s_q'}=\ob{b_p}{s_p}$. Thus $\ob{a_j}{r_j}\subseteq\ob{b_q'}{s_q'}$ and equality holds if $\cb{b_q'}{b_q'}\cap K_n=\emptyset$. If $n\neq m$ then we cannot have $p=k$ or $p=\ell$ since $\cb bs\subseteq U_m$ and $U_n\cap U_m=\emptyset$. If $n=m$ and either $p=k$ or $p=\ell$, then there exists $q\in S_{B'}$ such that $\ob{b_q'}{s_q'}=\ob bs$, so that $\ob{a_j}{r_j}\subseteq\ob{b_q'}{s_q'}$ and $\cb{b_q'}{s_q'}\cap K_n\neq\emptyset$. Moreover, $\cb{b_q'}{s_q'}\cap U_i=\emptyset$ for all $i\in I$ with $i\neq m$. It follows that $B'$ satisfies \ref{locally_partially_above}(d) and hence $B'\in\La C$. This completes the proof.
\end{proof}

Similar geometric reasoning and induction shows that, under the conditions of the lemma, given $n_1,\dotsc,n_p\in S_A$ and $m\in I$ such that
\[
\cb{a_{n_1}}{r_{n_1}}\cap K_m\neq\emptyset\quad\text{and}\quad \cb{a_{n_{j-1}}}{r_{n_{j-1}}}\cap\cb{a_{n_j}}{r_{n_j}}\neq\emptyset
\]for $j=2,\dotsc,p$ we have $\cb{b_{n_j}}{r_{n_j}}\subseteq U_m$ for each $j=1,\dotsc,p$.

We are now ready to prove the controlled classicalisation theorem.

\begin{theorem}\label{seqlocclassthm}
Let $I\subseteq\N$ be non-empty. Let $(K_n)_{n\in I}$ be a collection of compact plane sets and let $(M_n)_{n\in I}$ be a collection of positive real numbers. For each $n\in I,$ let $U_n:=U(K_n,M_n)$. Suppose that $U_k\subseteq\ob{a_0}{r_0}$ and $\rho_{U_k}(A)<M_k/2$ for all $k\in I$ and suppose that $U_k\cap U_\ell=\emptyset$ for all distinct $k,\ell\in I$. Let $\mathcal C$ be the controlling collection $((K_n,U_n))_{n\in I}$ and suppose $E(A)\subseteq F(\mathcal C)$. Then there exists $B=((b_n,s_n))\in\LM C$ such that $X_B\setminus V(\mathcal C)=X_A\setminus V(\mathcal C)$ and $B$ is classical.
\end{theorem}
\begin{proof}
We know that $\LM C$ is non-empty and compact so $\delta_2$ obtains its minimum on $\LM C$. Let $B\in\LM C$ such that $\delta_2$ is minimised on $\LM C$ at $B$.
We first show that $\cb{b_k}{s_k}\subseteq\ob{b_0}{s_0}$ for all $k\in S_B$. Let $C$ be the complement of the disk $\ob{a_0}{r_0}=\ob{b_0}{s_0}$. Let $k\in S_B$ and assume, for contradiction, that $C\cap\cb{b_k}{s_k}\neq\emptyset$. If there exists $u\in S_A$ such that $\ob{a_u}{r_u}=\ob{b_k}{s_k}$ then
\[
\emptyset\neq\cb{b_k}{s_k}\cap C=\cb{a_u}{r_u}\cap C\subseteq C\cap E(A)=\emptyset,
\]
which is impossible. Otherwise, by Lemma \ref{diskscorrespondence}, there exists $u\in S_A$ with $\ob{a_u}{r_u}\subseteq\ob{b_k}{s_k}$. Since $B\in \LM C$, it follows that there exists $m\in I$ such that $\cb{b_k}{s_k}\cap K_m\neq\emptyset$, and so $\cb{b_k}{s_k}\subseteq U_{m}\subseteq\ob{a_0}{r_0}=\ob{b_0}{s_0}$, which contradicts the fact that $C\cap\cb{b_k}{s_k}\neq\emptyset$.

We must now show that there do not exist distinct $k,\ell\in S_B$ such that $\cb{b_k}{s_k}\cap\cb{b_\ell}{s_\ell}\neq\emptyset$. Suppose, for contradiction, that such a pair exists. If $\cb{b_k}{s_k}\cap F(\mathcal C)=\emptyset$ and $\cb{b_\ell}{s_\ell}\cap F(\mathcal C)=\emptyset$ then there exists $u,v\in S_A$ with $\ob{a_u}{r_u}=\ob{b_k}{s_k}$ and $\ob{a_v}{r_v}=\ob{b_\ell}{s_\ell}$, which is a contradiction since $E(A)\subseteq F(\mathcal C)$. Thus at least one of these disks has non-empty intersection with at least one compact set $K_m$.

We may assume, without loss of generality, that $\ob{b_k}{s_k}\cap K_m\neq\emptyset$ for some $m\in I$. It follows that $s_k,s_\ell<M_m/2$ and $\cb{b_k}{s_k}\subseteq U_{m}$ and $\cb{b_\ell}{s_\ell}\cap U_{m}\neq\emptyset$. Let $\ob bs$ be the open disk obtained by an application of Lemma \ref{combinedisks} to the disks $\ob{b_k}{s_k}$ and $\ob{b_\ell}{s_\ell}$. Then, by Lemma \ref{newinLforseq}, the abstract Swiss cheese $B'\in\La C$ obtained by replacing the disks $\ob{b_k}{s_k}$ and $\ob{b_\ell}{s_\ell}$ with $\ob{b}{s}$ has either $\delta_1(B')>\delta_1(B)$ or $\delta_1(B')=\delta_1(B)$ and $\delta_2(B')<\delta_2(B)$. Both of these cases are impossible since we assumed that $\delta_1$ was maximised on $B$ and $\delta_2$ was minimised on $B$. It follows that no such pair $k,\ell$ can exist and hence $B$ is classical.

It remains to show that $X_B\setminus V(\mathcal C)=X_A\setminus V(\mathcal C)$. Note that $B\in\La C$ so $X_B\subseteq X_A$, thus $X_B\setminus V(\mathcal C)\subseteq X_A\setminus V(\mathcal C)$. Let $U_A:=(\C\setminus X_A)\cup V(\mathcal C)$ and $U_B:=(\C\setminus X_B)\cup V(\mathcal C)$. Let $z\in U_B$, we show that $z\in U_A$. If $z$ is outside of $\cb{b_0}{s_0}$ then it is also outside of $\cb{a_0}{r_0}$ since the closed balls are the same. If $z$ is in $\cb{b_0}{s_0}$, there exists $k\in S_B$ such that $z\in\ob{b_k}{s_k}$. Note that $\cb{b_k}{s_k}\cap F(\mathcal C)=\emptyset$, otherwise $\cb{b_k}{s_k}\subseteq V(\mathcal C)$. By Lemma \ref{diskscorrespondence}, there exists $\ell\in S_A$ such that $\ob{a_\ell}{r_\ell}=\ob{b_k}{s_k}$. Thus $z\in\ob{a_\ell}{r_\ell}$ and $U_B\subseteq U_A$. It follows that $\C\setminus U_B\supseteq\C\setminus U_A$ and hence $X_B\setminus V(\mathcal C)=X_A\setminus V(\mathcal C)$.
\end{proof}

Note here that the classical, abstract Swiss cheese $B$ obtained from this theorem is an element of $\LM C$ and therefore satisfies properties (a)-(d) of Definition \ref{locally_partially_above}, and the conclusion of Lemma \ref{diskscorrespondence} holds for $B$. Note also that, in contrast to the Feinstein-Heath classicalisation theorem, $\delta_1(B)$ may be negative here. We can obtain similar results using transfinite induction.

Taking $I$ to have just one element in Theorem \ref{seqlocclassthm}, we obtain the following corollary, which we use in Section \ref{applicationoflocal}.

\begin{corollary}\label{controlledonecorollary}
Let $K$ be a compact plane set and let $M$ be a positive real number. Let $U=U(K,M)$ and let $\mathcal C$ be the controlling pair $(K,U)$. Suppose that $\rho_{U}(A)<M$ and $E(A)\subseteq K$. Then there exists $B=((b_n,s_n))\in\LM C$ such that $X_B\setminus U=X_A\setminus U$ and $B$ is classical.
\end{corollary}

In Section \ref{applicationoflocal} we give an application of controlled classicalisation to construct an example of a classical Swiss cheese set $X$ such that $R(X)$ is regular and admits a non-degenerate bounded point derivation of infinite order, which improves the example constructed by O'Farrell \cite{o1979regular}. First we need to discuss annular classicalisation and discuss regularity of $R(X)$.

\section{Annular classicalisation}
\label{secannu}
In this section we give some results about Swiss cheese like sets obtained by deleting open disks from a closed annulus, rather than a closed disk. If $K$ is a closed annulus in the plane, we can write $K=\cb{a_0}{r_0}\setminus\ob{a_1}{r_1}$ for some $a_0=a_1\in\C$ and $r_0>r_1>0$ real. We say an abstract Swiss cheese $A=((a_n,r_n))$ is {\em annular} if $a_0=a_1$ and $0<r_1<r_0$ and let $K_A$ denote the annulus $\cb{a_0}{r_0}\setminus\ob{a_1}{r_1}$. We shall usually omit `abstract' from the statement {\em $A$ is an annular abstract Swiss cheese}.

\begin{lemma}\label{annupullin}
Let $a\in\C$ and $r_0>r_1>0$ and let $K:=\cb{a}{r_0}\setminus\ob{a}{r_1}$. Let $b\in\C$ and $0<s<(r_0-r_1)/2$ such that $\cb bs\cap\overline{\C\setminus K}\neq\emptyset$. Then there exists $r_0',r_1'>0$ such that $K':=\cb{a}{r_0'}\setminus\ob{a}{r_1'}\subseteq K$ with $K'\cap\ob bs=\emptyset$ and $r_0'-r_1'\geq r_0-r_1-2s$.
\end{lemma}
\begin{proof}
Set $D=\ob bs$. If $D\subseteq\C\setminus K$ then there is nothing to prove so suppose not. Since $s<(r_0-r_1)/2$ there are only two possible cases. We must have either $\overline D\cap\cb a{r_1}\neq\emptyset$ or $D\cap\C\setminus\ob{a}{r_0}\neq\emptyset$.

In the first case, where $\bar D\cap\cb a{r_1}\neq\emptyset$, let $r_0'=r_0$ and $r_1'=\abs{b-a}+s$. We have $\abs{b-a}>r_1-s$ and $\abs{b-a}\leq r_1+s$. Hence $r_1'>r_1-s+s=r_1$ and $r_1'\leq r_1+2s<r_1+r_0-r_1=r_0$ and
\[
r_0'-r_1'=r_0-(\abs{b-a}+s)\geq r_0-s-r_1-s=r_0-r_1-2s.
\]
Since for each $z\in D$ we have $\abs{b-a}-s<\abs{z-a}<\abs{b-a}+s$ it follows immediately that $D\subseteq\C\setminus K$.

In the second case, where $D\cap\C\setminus\ob{a}{r_0}\neq\emptyset$, let $r_0'=\abs{b-a}-s$ and $r_1'=r_1$. We have $\abs{b-a}<r_0+s$ and $\abs{b-a}\geq r_0-s$. Hence $r_0'<r_0+s-s=r_0$ and
\[
r_0'>r_0-s-s>r_0-(r_0-r_1)=r_1
\]
and so
\[r_0'-r_1'=\abs{b-a}-s-r_1\geq r_0-r_1-2s.\]
Similarly, for all $z\in D$ we have $\abs{b-a}-s<\abs{z-a}<\abs{b-a}+s$ and so $D\subseteq\C\setminus K$. This completes the proof.
\end{proof}

\begin{definition}
The {\em annular radius sum function} $\ar:\scs\to[0,\infty]$ is defined by
\[
\ar(A):=\sum\limits_{n=2}^\infty r_n \qquad(A=((a_n,r_n))\in\scs),
\]
and the {\em annular discrepancy function} $\ad:\scs\to [-\infty,\infty)$ is given by
\[
\ad(A)=r_0-r_1-2\ar(A)\qquad (A=((a_n,r_n))\in\scs).
\]
\end{definition}
Note that if $\ad(B)>0$ then $r_0>r_1$. We aim to prove an analogue of the Feinstein-Heath classicalisation theorem (Theorem \ref{feinheaththm}) for annular Swiss cheeses by constructing a suitable compact subset of $\scs$.

It is easy for the reader to check that the following analogue of Lemma \ref{nonredundantcheese} holds for annular Swiss cheeses.

\begin{lemma}\label{annularredfreeequiv}
Let $A$ be an annular Swiss cheese with $\ar(A)<\infty$. Then there exists an annular Swiss cheese $B=((b_n,s_n))$ {with the following properties$:$ $\ar(B)\leq\ar(A),$ $X_B=X_A$ and $K_B=K_A;$ $\mu(B)<\infty;$ the sequence $(s_n)_{n\geq2}$ is non-increasing$;$ for each $j\in S_B\setminus\{1\},$ we have $\ob{b_j}{s_j}\cap K_B\neq\emptyset$ and $\ob{b_j}{s_j}\nsubseteq\ob{b_k}{s_k}$ for all $k\in S_B\setminus\{1,j\}$}. Moreover, for each $E\subseteq\C,$ we have $\rho_E(B)\leq\rho_E(A)$.
\end{lemma}

Note that, in the previous lemma, $K_B=K_A$ and $\ar(B)\leq\ar(A)$ together imply that $\ad(B)\geq\ad(A)$.

\sskip

For the rest of this section, let $A=((a_n,r_n))$ be an annular Swiss cheese with $\ad(A)>0$, such that $\mu(A)<\infty$ and $(r_n)_{n=2}^\infty$ is non-increasing.

\begin{lemma}\label{compactannu}
Let $\An$ be the family of all $B=((b_n,s_n))\in\scs$ such that
\begin{enumerate}
 \item the sequence $(s_n)_{n\geq 2}$ is non-increasing$,$
 \item $\ar(B)\leq \ar(A),$
 \item $\mu(B)\leq\mu(A),$
 \item $B$ is partially above $A,$ and
 \item $b_0=b_1=a_0,$ and $r_0\geq s_0\geq s_1\geq r_1$.
\end{enumerate}
Then $\An$ is compact in $\scs,$ each abstract Swiss cheese $B\in\An$ with $\ad(B)>0$ is annular. Moreover$,$ the function $\ad|_{\An}:\An\to\R$ is upper semicontinuous and the function $\delta_2|_{\An}:\An\to\R$ is continuous.
\end{lemma}
\begin{proof}
It is easy to see that the family $\An$ is pointwise bounded by properties (b),(c) and (e) so it remains only to prove that $\An$ is closed. For each $m\in\No$, let $A^{(m)}=((a_n^{(m)},r_n^{(m)}))_{n=0}^\infty\in\An$ and suppose that $A^{(m)}\to B\in\scs$ as $m\to\infty$. It is clear that $B$ satisfies (a)-(d) (as in the proof of Lemma \ref{Closed_subset}). Since convergence is pointwise, we have $b_0=a_0$ and $b_1=a_1$. Since $A$ was annular, it follows that $b_0=b_1$.

Since each $A^{(m)}\in\An$ we have $r_0\geq r_0^{(m)}\geq r_1^{(m)}\geq r_1$, by taking $m\to \infty$ we have
\[
r_0\geq s_0\geq s_1\geq r_1.
\]
Hence $\An$ is closed and pointwise bounded and is therefore compact by Tychonoff's theorem.

Let $B=((b_n,s_n))\in\An$ with $\ad(B)>0.$ Then we have $b_0=b_1$ and $\ad(B)>0$ and this implies that $s_0>s_1$ and it follows that $B$ is annular.

The proof that $\ad$ is upper semicontinuous is an immediate consequence of Fatou's lemma for series, similar to the upper semicontinuity of $\delta_1$.

To prove that the restriction of $\delta_2$ to $\An$ is continuous note that, for $n\in\N$ with $n\geq 2,$ we have $s^2_n\leq\ar(B)^2/n^2$ for each $B=((b_n,s_n))\in\An$. The result then follows from the dominated convergence theorem as in the proof of Lemma \ref{discrepancy_function}.
\end{proof}

It is clear that $A\in\An$ and so $\An$ is non-empty. For all $B\in\An$ we also have $X_B\subseteq X_A$. We require one additional lemma before we prove the main theorem.

\begin{lemma}\label{annusuccstep}
Let $\An$ be as in Lemma $\ref{compactannu}$. Let $B=((b_n,s_n))\in\An$ be an annular Swiss cheese such that $\ad(B)\geq\ad(A)$. Suppose there exists $k\in S_B\setminus \{1\}$ such that $\cb{b_k}{s_k}\cap\overline{\C\setminus K_B}\neq\emptyset$. Then there exists $B'=((b_n',s_n'))\in \An$ with $\ad(B')\geq\ad(B)$. Moreover$,$ if $\ad(B')=\ad(B)$ then $\delta_2(B')<\delta_2(B)$.
\end{lemma}
\begin{proof}
Let $b'_0=b'_1=b_0$. As in Lemma~\ref{annupullin}, we can find $s'_0>s'_1>0$ such that $K_{B'}:=\cb{b'_0}{s'_0}\setminus \ob{b'_1}{s'_1} \subseteq K_B$, $K_{B'}\cap \ob{b_k}{s_k}=\emptyset$ and
\[
s'_0-s'_1\geq s_0-s_1-2s_k.
\]
Let $b'_\ell = b_\ell$ and $s'_\ell = s_\ell$ if $2\leq \ell <k$, $b'_\ell = b_{\ell+1}$ and $s'_\ell = s_{\ell+1}$ if $k< \ell$, we obtain an abstract Swiss cheese $B'=((b_n',s_n'))$.

From construction we see $B'$ satisfies Properties (a),(c) and (e).
We have
\[
\ad(B')=s_0'-s_1'-2\sum\limits_{n=2}^\infty s_n'\geq s_0-s_1-2s_k-2\sum\limits_{n=2}^\infty s_n+2s_k=\ad(B).
\]
Since $s_0'\leq s_0$ and $s_1'\geq s_1$ we must have $\ar(B')\leq\ar(B)\leq\ar(A)$, so (b) is satisfied.

We now show that $B'$ is partially above $A$. Fix $j\in S_A$. If $\ob{a_j}{s_j}$ lies in the complement of $\ob{b_0}{s_0}$, then it lies in the complement of $\ob{b_0'}{s_0'}$ and if $\ob{a_j}{s_j}\subseteq\ob{b_1}{s_1}$ then $\ob{a_j}{s_j}\subseteq\ob{b_1'}{s_1'}$. Suppose there exists $m\in S_B$ such that $\ob{a_j}{s_j}\subseteq\ob{b_m}{s_m}$. If $m\neq k$ there exists $\ell\in S_{B'}$ such that $\ob{b_{\ell}'}{s_\ell'}=\ob{b_m}{s_m}$, and so $\ob{a_j}{s_j}\subseteq\ob{b_\ell'}{s_\ell'}$. If $m=k$ then either $\ob{a_j}{r_j}\subseteq\ob{b_1'}{s_1'}$ or $\ob{a_j}{s_j}$ lies in the complement of $\ob{b_0'}{s_0'}$. It follows that $B'$ is partially above $A$, and satisfies $4$ and hence $B'\in\An$. Since we have $\ad(B')\geq\ad(A)>0$, it follows that $B'$ is annular.

It remains to show that if $\ad(B')=\ad(B)$ then $\delta_2(B')<\delta_2(B)$. Assume that $\ad(B')=\ad(B)$. Then either $s_0=s_0'+2s_k$ or $s_1'=s_1+2s_k$. In the first case we have $(s_0')^2<s_0^2-4s_k^2<s_0^2-s_k^2$ and in the second case we have $(s_1')^2>s_1^2+s_k^2$. In the first case we have $s_0^2>(s_0')^2+s_k^2$, and in the second case we have $(s_1')^2>s_1^2+s_k^2$. In either case, we have $\delta_2(B')<\delta_2(B)$. This completes the proof.
\end{proof}

Note that, as for arbitrary abstract Swiss cheeses, if $B$ is a semiclassical, annular Swiss cheese then $\pi\delta_2(B)$ is the area of $X_B$.

\begin{theorem}\label{annuclass}
Let $\An$ be as in Lemma $\ref{compactannu}$. Then there exists a classical$,$ annular Swiss cheese $B=((b_n,s_n))\in\An$ such that $\ad(B)\geq\ad(A)$ and $X_B\subseteq X_A$. Moreover$,$ $r_0-2\ar(A)\leq s_0\leq r_0$ and $r_1\leq s_1\leq r_1+2\ar(A)$.
\end{theorem}
\begin{proof}
Since $\ad$ is upper semicontinuous on $\An$ and $\An$ is compact and non-empty, it follows that $\ad$ achieves its maximum on $\An$. Let $\An_1$ denote the non-empty, compact subset of $\An$ on which $\ad$ is maximised. Then $\delta_2$, which is continuous on $\An_1$, achieves its minimum. Let $\An_2$ denote the non-empty, compact subset of $\An_1$ on which $\delta_2$ is minimised and let $B=((b_n,s_n))\in\An_2$.

Since $\ad(B)\geq\ad(A)>0$ it follows that $B$ is annular and $X_B\subseteq X_A$. Suppose, for contradiction, that $B$ is non-classical. There are two possible cases.

First suppose that there are $k,\ell\in S_B\setminus\{1\}$ with $k>\ell$ such that $k,\ell\in S_B$ and $\cb{b_k}{s_k}\cap\cb{b_\ell}{s_\ell}\neq\emptyset$. Then, by Lemma \ref{combinedisks} there exists $b\in\C$ and $s>0$ such that
\[
\ob{b_k}{s_k}\cup\ob{b_\ell}{s_\ell}\subseteq\ob bs
\]
and $s\leq s_k+s_\ell$. Let $B'=((b_n',s_n'))$ be the abstract Swiss cheese obtained by deleting the disks at indices $k,\ell$ from $B$ and inserting the disk $\ob bs$ at the first index in $\mathbb N\setminus \{1\}$ such that $(s_n')_{n=2}^\infty$ is non-increasing. It is easy to see that $B'\in\An$ and
\begin{equation}\label{annuclasseqn}
\ar(B)\geq \ar(B)-s_k-s_\ell+s=\ar(B'),
\end{equation}
so that $\ad(B')\geq\ad(B)$. By the maximality of $\ad(B)$, equality must hold here and in \eqref{annuclasseqn}. Thus $s=s_k+s_\ell$ and $s^2=(s_k+s_\ell)^2> s_k^2+s_\ell^2$ so that $\delta_2(B')<\delta_2(B)$. This contradicts the minimality of $\delta_2(B)$. It follows that no such $k,\ell$ exist.

Now suppose there exists $k\in S_B\setminus\{1\}$ such that $\cb{b_k}{s_k}\cap\overline{\C\setminus K_B}\neq\emptyset$ and $s_k>0$. By Lemma \ref{annusuccstep} there exists an annular Swiss cheese $B'\in\An$ with $\ad(B')\geq\ad(B)$ such that, if $\ad(B')=\ad(B)$ then $\delta_2(B')<\delta_2(B)$. This is a contradiction, so no such $k$ can exist. It follows that $B$ is classical.

Since $B\in\An$, we have $r_0\geq s_0\geq s_1\geq r_1$. We also have
\[
s_0-s_1\geq\ad(B)\geq\ad(A)=r_0-r_1-2\ar(A)
\]
so that
\[
s_0\geq r_0-2\ar(A)-(r_1-s_1)\geq r_0-2\ar(A)
\]
and $s_1\leq r_1+2\ar(A)-(r_0-s_0)\leq r_1+2\ar(A)$. This completes the proof.
\end{proof}

\section{Regularity of $R(X)$}
\label{regularitysection}
Let $X$ be a compact plane set. We say that $R(X)$ is {\em regular} if, for all closed sets $E\subseteq X$ and points $x\in X\setminus E$, there exists a function $f\in R(X)$ such that $f(x)=1$ and $f(y)=0$ for all $y\in E$. We say that $R(X)$ is {\em normal} if, for each pair of disjoint closed sets $E,F\subseteq X$, there exists a function $f\in R(X)$ such that $f(x)=0$ for all $x\in E$ and $f(x)=1$ for all $x\in F$. It is standard that $R(X)$ is regular if and only if it is normal (see \cite[Proposition~4.1.18]{dales2000}).

In order to avoid ambiguity, we introduce the following notation to clarify in which topological space we are taking the interior. Let $X$ be a compact plane set and $E\subseteq X$. Then $\tint_X E$ denotes the interior of $E$ in the topological space $X$.

\begin{definition}
Let $X$ be a compact plane set$,$ and let $x\in X$. We denote by $M_x$ the ideal of all functions in $R(X)$ which vanish at $x$. We denote by $J_x$ the ideal of all functions in $R(X)$ which vanish on a neighbourhood of $x$. We say $x$ is an $R$-point for $R(X)$ if$,$ for all $y\in X$ with $y\neq x,$ we have $J_x\nsubseteq M_y$.
\end{definition}

It is standard that $R(X)$ is regular if and only if every point $x\in X$ is an $R$-point of $R(X)$. The following proposition is a special case of \cite[Corollary~4.7]{feinstein2000}.

\begin{proposition}\label{uncountablenonrpoint}
Let $X$ be a compact plane set such that $R(X)$ is not regular. Let $E$ denote the set of non-$R$-points for $R(X)$. Then $E$ contains a non-empty perfect subset. In particular$,$ $E$ is uncountable.
\end{proposition}

Our classicalisation theorems involve finding ``good'' compact subsets of a given compact plane set. The following proposition, stated in \cite{feinheath2010}, lists some properties of $R(X)$ which are inherited when a subset of $X$ is considered.

\begin{proposition}\label{subset_property_inheritence}
Let $X$ and $Y$ be compact plane sets with $Y\subseteq X$. Then$:$
\begin{enumerate}
  \item if $R(X) = C(X)$ then $R(Y) = C(Y);$
  \item if $R(X)$ does not have any non-zero bounded point derivations then neither does $R(Y);$
  \item if $R(X)$ is regular then so is $R(Y)$.
\end{enumerate}
\end{proposition}

In this section, we prove some results about regularity of $R(X)$ which we shall require for the construction in the final section. The following proposition is essentially \cite[Corollary~II.10.3]{gamelin1984}.

\begin{proposition}\label{localisation theorem}
Let $X$ be a compact plane set and let $f\in C(X)$. Suppose that for each $x\in X$ there is a closed neighbourhood $N_x$ of $x$ in $X$ such that $f|_{N_x}\in R(N_x)$. Then $f\in R(X)$.
\end{proposition}

We shall require the following theorem.

\begin{theorem}\label{regularityunion}
Let $X$ be a compact plane set and let $E$ be a countable subset of $X$. Let $(X_\alpha)$ be a family of compact plane sets such that $R(X_\alpha)$ is regular for all $\alpha$ and $\bigcup_{\alpha}\tint_X{(X\cap X_\alpha)}\supseteq X\setminus E$. Then $R(X)$ is regular.
\end{theorem}
\begin{proof}
We first show that every point in $X\setminus E$ is an $R$-point for $R(X)$. Let $x\in X\setminus E$ and $y\in X$ with $x\neq y$. Then there exists $\alpha$ and $r>0$ such that $\cb{x}{r}\cap X\subseteq X_\alpha$. Let $\delta<r/3$ such that $\abs{x-y}>2\delta$. Let $F$ denote the complement of $X\cap\ob{x}{2\delta}$ in $\cb xr\cap X$. Since $X\cap\cb xr\subseteq X_\alpha$, $R(X\cap\cb xr)$ is regular by Proposition \ref{subset_property_inheritence} (and hence normal). Thus there exists a function $g\in R(X\cap \cb xr)$ with $g(z)=0$ for all $z\in X\cap \cb x\delta$ and $g(z)=1$ for all $z\in F$. Extend $g$ to a function $f\in C(X)$ by setting $f(z)=g(z)$ for all $z\in (\cb xr\cap X)$ and $f(z)=1$ for all $z\in X\setminus\cb xr$. Clearly $f$ satisfies the conditions of Proposition \ref{localisation theorem}, so $f\in R(X)$. By our choice of $\delta$, we have $x\in U$ and $y\in F$ so $f$ vanishes on a neighbourhood of $x$ and $f(y)=1$, so $x$ is an $R$-point for $R(X)$.

It follows that $R(X)$ has at most countably many non-$R$-points. So, by Proposition \ref{uncountablenonrpoint}, $R(X)$ is regular.
\end{proof}

Note that we do not assume that $X_\alpha\subseteq X$. However, replacing $X_\alpha$ by $X\cap X_\alpha$ does not alter the result.
We obtain the following corollaries.

\begin{corollary}\label{onepointregular}
Let $X$ be a compact plane set and $x_0\in X$. Let $(X_\alpha)$ be a family of compact plane sets such that $\bigcup_\alpha\tint_X(X\cap X_\alpha)=X\setminus\{x_0\}$ and $R(X_\alpha)$ is regular for all $\alpha$. Then $R(X)$ is regular.
\end{corollary}

\begin{corollary}
Let $X_1,X_2$ be compact plane sets such that $X_1\cap X_2$ is countable. If $R(X_1)$ and $R(X_2)$ are regular then $R(X_1\cup X_2)$ is regular.
\end{corollary}

\section{Classicalisation of an example of O'Farrell}
\label{applicationoflocal}
In this section we see an application of the results of Sections \ref{seclocal}-\ref{regularitysection}. In \cite{o1979regular}, O'Farrell modified the construction of McKissick \cite{mckissick1963nontrivial} to construct a Swiss cheese set $X$ such that $R(X)$ is regular and admits a non-degenerate bounded point derivation of infinite order (defined below). However, this Swiss cheese set is not necessarily classical.

\begin{definition}
Let $X$ be a compact plane set and let $x\in X$.
A {\em point derivation of order} $n\in \N$ (respectively, $\infty$) at $x$ (on $R(X)$) is a sequence $d_0,d_1,\dotsc$ of linear functionals with $d_0=\ec x,$ the evaluation character at $x,$ satisfying
\[
d_j(fg) = \sum\limits_{k=0}^jd_k(f)d_{j-k}(g)\qquad(f,g\in R(X)),
\]
for all $j=1,2,\dotsc,n$ (respectively, $j=1,2,\dotsc$).

Let $\mathbf{d}=(d_j)_{j=0}^n$ be a point derivation of order $n$ at $x$ (where we include the possibility that $n=\infty$ when $(d_j)$ is a point derivation of infinite order).
We say that $\mathbf{d}$ is \emph{bounded} if $d_j$ is a bounded linear functional for each $j$ with $j\leq n$ (respectively, all $j$).
We say that $\mathbf{d}$ is {\em non-degenerate} if $d_1\neq 0$.
\end{definition}

{We refer the reader to \cite{dales1977} (especially Lemma 2.1 and p.~170) for further details, related results and comments concerning non-degenerate higher point derivations.}

Following our general scheme of classicalisation, we construct a classical Swiss cheese set $X$ such that $R(X)$ is regular and admits a non-degenerate bounded point derivation of infinite order at one of the points of $X$.

The following proposition is an immediate corollary of Proposition \ref{subset_property_inheritence} and the result of McKissick \cite[Proposition~1.10]{feinheath2010} (see also \cite{mckissick1963nontrivial} and \cite{o1979regular}).

\begin{proposition}\label{regularannu}
Let $b_0=b_1\in\C,$ let $s_0>s_1>0,$ and let $\varepsilon>0$. Then there exists an annular Swiss cheese $A=((a_n,r_n))$ with $\ar(A)<\varepsilon,$ $a_j=b_j$ and $r_j=s_j$ for $j=0,1,$ and such that $R(X_A)$ is regular.
\end{proposition}

We now use a sequence of lemmas to show that we can construct a classical annular Swiss cheese with the same properties as those in Proposition \ref{regularannu}.

\begin{lemma}\label{annuapproxlem}
Let $\lambda_0>\lambda_1>0$ and $\varepsilon,\eta>0$ be given and let $a\in\C$. There exists a classical$,$ annular Swiss cheese $B=((b_n,s_n))$ with $b_0=b_1=a$ such that $\ar(B)<\varepsilon,$ $\lambda_0\leq s_0\leq\lambda_0+\eta$ and $\lambda_1-\eta\leq s_1\leq\lambda_1$ such that $R(X_B)$ is regular.
\end{lemma}
\begin{proof}
We may assume that $\eta<\lambda_1$ and $\varepsilon\leq\eta/2$. Let $A=((a_n,r_n))$ be an abstract Swiss cheese obtained from Proposition \ref{regularannu} with $a_0=a_1=a$, $r_0=\lambda_0+\eta$, $r_1=\lambda_1=\eta$ and $\ar(A)<\varepsilon$ and such that $R(X_A)$ is regular. By Lemma \ref{annularredfreeequiv}, we may assume that the sequence $(r_n)_{n=2}^\infty$ is non-increasing. Apply Theorem \ref{annuclass} to the abstract Swiss cheese $A$ to obtain a classical, annular Swiss cheese $B=((b_n,s_n))$ with
\[
b_0=b_1=a_0=a_1=a,\quad r_0-2\varepsilon\leq s_0\leq r_0,\quad\text{and}\quad r_1\leq s_1\leq r_1+2\varepsilon,
\]
such that $\ar(B)\leq\ar(A)$ and $X_B\subseteq X_A$. By Lemma \ref{subset_property_inheritence}, $R(X_B)$ is regular. Since $2\varepsilon\leq\eta$, we have $\lambda_0\leq s_0\leq \lambda_0+\eta$ and $\lambda_1-\eta\leq s_0\leq \lambda_1$. This completes the proof.
\end{proof}

By instead taking $r_0=\lambda_0$ and $r_1=\lambda_1$ in the proof of the previous lemma we see that we could also approximate the desired annulus with a smaller annulus, rather than a larger annulus as in Lemma \ref{annuapproxlem}.

In the next lemma, we see how to obtain a classical, annular Swiss cheese $A$ such that $R(X_B)$ is regular with one of the first two radii specified exactly while prescribing tight bounds on the other.

\begin{lemma}\label{onefixoneapproxannu}
Let $\lambda_0>\lambda_1>0$ and $\varepsilon,\eta>0$ be given and let $a\in\C$.
\begin{enumerate}
 \item There exists a classical$,$ annular Swiss cheese $B^{(1)}=((b_n^{(1)},s_n^{(1)}))$ with $s_0^{(1)}=\lambda_0,$ $\lambda_1-\eta\leq s_1^{(1)}\leq\lambda_1$ and $b_0^{(1)}=b_1^{(1)}=a$ such that $R(X_{B^{(1)}})$ is regular and $\ar(B^{(1)})<\varepsilon$.
 \item There exists a classical$,$ annular Swiss cheese $B^{(2)}=((b_n^{(2)},s_n^{(2)}))$ with $b_0^{(2)}=b_1^{(2)}=a$ such that $\ar(B^{(2)})<\varepsilon,$ $\lambda_0\leq s_0^{(2)}\leq\lambda_0+\eta$ and $s_1^{(2)}=\lambda_1$ and  such that $R(X_{B^{(2)}})$ is regular.
\end{enumerate}
\end{lemma}
\begin{proof}
We prove (b); the proof of (a) is similar but easier. We may assume, without loss of generality, that $a=0$.
Let $\gamma\in(0,\lambda_1)$ to be chosen later. Apply Lemma \ref{annuapproxlem} to obtain a classical, annular Swiss cheese $A=((a_n,r_n))$ such that $\ar(A)<\gamma/2$, $\lambda_0\leq r_0\leq\lambda_0+\gamma$ and $\lambda_1-\gamma\leq r_1\leq\lambda_1$ and such that $R(X_A)$ is regular.

For each $n\geq 0$ let $b_n^{(2)}:=\lambda_1 a_n/r_1$ and $s_n^{(2)}:=\lambda_1 r_n/r_1$. Then
\[
\sum\limits_{n=2}^\infty s_n^{(2)}\leq\frac{\lambda_1}{\lambda_1-\gamma}\sum\limits_{n=2}^\infty r_n<\frac{\lambda_1}{\lambda_1-\gamma}\frac\gamma2.
\]
Set $M_\gamma:=\lambda_1/(\lambda_1-\gamma)>1$. We have $s_1^{(2)}=\lambda_1$ and $s_0^{(2)}\geq\lambda_0$ so it remains to show that $s_0^{(2)}\leq\lambda_0+\eta$ provided that $\gamma$ is sufficiently small. We have
\[
\lambda_0\leq s_0^{(2)}=\frac{\lambda_1}{r_1}r_0\leq M_\gamma r_0\leq M_\gamma(\lambda_0+\gamma).
\]
Since $M_\gamma\to 1$ as $\gamma\to 0$, if $\gamma$ is small enough then we have $\lambda_0\leq s_0^{(1)}\leq \lambda_0+\eta$ and $M_\gamma\gamma/2<\varepsilon$. Clearly $R(X_{B^{(2)}})$ is regular since $R(X_A)$ is regular. This completes the proof of the (b).
\end{proof}

Our final lemma shows that we can obtain this type of annular Swiss cheese with $s_0$ and $s_1$ precisely prescribed.

\begin{lemma}\label{exactPannu}
Let $\lambda_0>\lambda_1>0$ and $\varepsilon>0$ be given and let $a\in\C$. There exists a classical$,$ annular Swiss cheese $B=((b_n,s_n))$ with $b_0=b_1=a,$ $s_0=\lambda_0,$ $s_1=\lambda_1,$ $\ar(B)<\varepsilon$ and such that $R(X_B)$ is regular.
\end{lemma}
\begin{proof}
Let $\kappa=(\lambda_0+\lambda_1)/2$ and let $\eta>0$. By Lemma \ref{onefixoneapproxannu}(a) there exists a classical, annular Swiss cheese $A^{(1)}=((a_n^{(1)},r_n^{(1)}))$ with $a_0^{(1)}=a_1^{(1)}=a$, $\ar(A^{(1)})<\eta/16$, $r_0^{(1)}=\lambda_0$ and $\kappa-\eta/4\leq r_1^{(1)}\leq\kappa-\eta/8$ and such that $R(X_{A^{(1)}})$ is regular. By Lemma \ref{onefixoneapproxannu}(b) there exists a classical, annular Swiss cheese $A^{(2)}=((a_n^{(2)},r_n^{(2)}))$ with $a_0^{(2)}=a_1^{(2)}=a$, $\ar(A^{(2)})<\eta/16$, $r_1^{(2)}=\lambda_1$ and $\kappa\leq r_0^{(2)}\leq\kappa+\eta/4$ and such that $R(X_{A^{(2)}})$ is regular.

Let $(a_n^{(3)})_{n\geq 2}$ be a sequence containing all elements from the sequences $(a_n^{(1)})_{n\geq 2},$ and $(a_n^{(2)})_{n\geq 2}$ exactly once and let $(r_n^{(3)})_{n\geq 2}$ be the corresponding sequence containing all elements from the sequences $(r_n^{(1)})_{n\geq 2},$ and $(r_n^{(2)})_{n\geq 2}$ exactly once. Let $a_0^{(3)}:=a,$ $a_1^{(3)}:=a$ and $r_0^{(3)}:=\lambda_0,$ $r_1^{(3)}:=\lambda_1$ and let $A^{(3)}=((a_n^{(3)},r_n^{(3)}))$ be the corresponding annular Swiss cheese. Then
\[
\ar(A^{(3)})=\sum\limits_{n=2}^\infty r_n^{(3)}<\frac\eta{16}+\frac\eta{16}= \frac{\eta}{8}.
\]
Let $X:=X_{A^{(3)}}$, then we can easily check that $X=\tint_X X_{A^{(1)}}\cup\tint_X X_{A^{(2)}}$ so, by Theorem \ref{regularityunion}, $R(X)$ is regular.

Choose $\eta>0$ small enough so that we have $\eta<(\lambda_0-\lambda_1)/4$, $\eta/2<\varepsilon$ and $\eta<\lambda_1$. Let $K:=\{z\in\C:\kappa-\eta/4\leq\abs{z}\leq\kappa+\eta/4\}$ and let $M:=\eta/4$.
Let $A=((a_n,r_n))\in\noninc$ be obtained by applying Lemma \ref{annularredfreeequiv}. Then, for each open $U\subseteq\C,$ $\rho_U(A)\leq\rho_U(A^{(3)})$. It is now easy to see that $A,K$ and $M$ satisfy the conditions of Corollary \ref{controlledonecorollary}. Note that $U:=U(K,M)$ has $U\cap \cb{a_1}{r_1}=\emptyset$. Note that $X_A=X=X_{A^{(3)}}$.

Let $B=((b_n,s_n))$ be the classical abstract Swiss cheese obtained by applying Corollary \ref{controlledonecorollary} to $A,K$ and $M$. Then $B\in\LM C$, where $\mathcal C$ is the controlling pair $(K,U)$. Thus, by Lemma \ref{diskscorrespondence}, there exists $\ell\in S_B$ such that $\ob{b_\ell}{s_\ell}=\ob{a_1}{r_1}$. Since the sequence $(s_n)_{n\geq 1}$ is non-increasing and $\ar(A)<r_1$, it follows that $\ell=1$. It follows that $B$ is annular and has $s_0=\lambda_0$, $s_1=\lambda_1$ and $\ar(B)<\varepsilon$. Since $R(X_{A})$ is regular, by Proposition \ref{subset_property_inheritence}, $R(X_B)$ is regular. This completes the proof.
\end{proof}

We are now ready to construct a classical Swiss cheese set $X$ such that $R(X)$ is regular and admits a non-degenerate bounded point derivation of infinite order.

\begin{theorem}\label{zerocheese}
Let $\varepsilon>0$. Then there exists a classical abstract Swiss cheese $B=((b_n,s_n))$ with $0\in X_B$ and $\rho(B)<\varepsilon$ and such that $R(X_B)$ is regular and admits a non-degenerate bounded point derivation of infinite order at $0$.
\end{theorem}
\begin{proof}
We may assume that $\varepsilon<2^{-5}$. For each $n\in\N$, let $\gamma_n=(2n)^{-n}\varepsilon$. Note that $\sum_{n=1}^\infty \gamma_n<\varepsilon$. Let $A^{(1)}=((a_n^{(1)},r_n^{(1)}))$ be a classical, annular Swiss cheese, given by Lemma \ref{exactPannu}, with $a_0^{(1)}=0$, $\ar(A^{(1)})<\gamma_12^{-3}$, $r_0^{(1)}=1$ and $r_1^{(1)}=2^{-1}$ such that $R(X_1)$ is regular, where $X_1:=X_{A^{(1)}}$. For each $m\geq 2$ let $A^{(m)}=((a_n^{(m)},r_n^{(m)}))$ be a classical, annular Swiss cheese, given by Lemma \ref{exactPannu}, with $a_0^{(m)}=0$, $\ar(A^{(m)})<\gamma_m2^{-m-2}$, $r_0^{(m)}=(33/32)2^{1-m}$ and $r_1^{(m)}=2^{-m}$ such that $R(X_m)$ is regular, where $X_m:=X_{A^{(m)}}$. Note that, for each $m\in\N$, since $A^{(m)}$ is annular we have $a_1^{(m)}=0$.

For each $m\in\N$ we have
\begin{equation}\label{radsumeq}
\ar(A^{(m)})+\ar(A^{(m+1)})\leq\frac{\gamma_m}{2^{m+2}}+\frac{\gamma_{m+1}}{2^{m+3}} <\frac{3}{2}\frac{\gamma_m}{2^{m+2}}.
\end{equation}

The set $\{(a_n^{(m)},r_n^{(m)}):m,n\in\N,n\geq 2\}$ is countably infinite so we may enumerate it as a sequence of pairs $(a_n,r_n)$, so that each pair occurs exactly once. Let $a_0:=0$ and $r_0:=1$ and let $A=((a_n,r_n))$ be the resulting abstract Swiss cheese. It is clear that $0\in X_{A}$. It is easy to check that
\[
\bigcup_{m=1}^\infty\tint_{X_A}X_m=X_A\setminus\{0\}
\]
so that $R(X_A)$ is regular by Corollary \ref{onepointregular}.

Using the notation for closed annuli from Section \ref{secannu}, for each $m\in\N$, let $W_m=K_{A^{(m)}}\cup K_{A^{{(m+1)}}}$. Then, by \eqref{radsumeq}, we see that
\begin{equation}\label{Aproperty1}
\rho_{W_m}(A)<(3/2)\gamma_m 2^{-m-2}.
\end{equation}
We also have
\begin{equation}\label{Aproperty2}
\rho(A)< \sum\limits_{m=1}^\infty2^{-m-2}\gamma_m<\varepsilon.
\end{equation}
By an application of Lemma \ref{nonredundantcheese}, we may assume that $A$ is redundancy-free and $A\in\mathcal N$ while preserving $X_A$, the regularity of $R(X_A)$ and the inequalities \eqref{Aproperty1} and \eqref{Aproperty2}.

For each $m\in\N$, let $K_m:=\{z\in\C:(15/16)2^{-m}\leq\abs{z}\leq(17/16)2^{-m}\}$ and $M_m=3\gamma_m/2^{m+2}$ so that, since $\gamma_m<1$, we have
\begin{equation}\label{Uplus2sum}
2^{-m-4}+M_m<2^{-m-4}(1+3\cdot 2^{-m})\leq(5/32)2^{-m}.
\end{equation}
For each $m\in\N$, define $U_m:=U(K_m,M_m)$ as in the statement of Theorem \ref{seqlocclassthm}, and let $\mathcal C$ be the controlling collection $((K_n,U_n))_{n\in\N}$. By \eqref{Uplus2sum}
\[\label{Uinannulus}
U_m\subseteq\bigg\{z\in\C:\frac{27}{32}2^{-m}\leq\abs{z}\leq\frac{37}{32}2^{-m}\bigg\}
\]
and hence $U_m\cap U_{n}=\emptyset$ for all $m,n\in\N$ with $m\neq n$. Clearly $E(A)\subseteq F(\mathcal C)$. Since $U_m\subseteq W_m$, for each $m\in\N$, we see that
\[
\rho_{U_m}(A)\leq\rho_{W_m}(A)<\frac{3}{2}\frac{\gamma_m}{2^{m+2}}
\]
This shows that the sequences $(K_n)_{n\geq 1}$ and $(M_n)_{n\geq 1}$ satisfy the conditions of Theorem \ref{seqlocclassthm}.

Applying Theorem \ref{seqlocclassthm} to $A$, $(K_n)_{n\geq 1}$ and $(M_n)_{n\geq 1}$, there exists an abstract Swiss cheese $B=((b_n,s_n))\in\LM C$ with $b_0=0$, $s_0=1$, $X_B\subseteq X_{A}$ and $\rho(B)<\varepsilon$ such that $B$ is classical and $X_B\setminus V(\mathcal C)=X_A\setminus V(\mathcal C)$; in particular $0\in X_B$. Also, we have $\rho_{U_m}(B)\leq\rho_{U_m}(A)$ for each $m\in\N$. By Proposition \ref{subset_property_inheritence}, $R(X_B)$ is regular.

For each $m\in\No,$ let $E_m:=\{z\in\C:(3/2)2^{-m-1}\leq\abs{z}\leq (3/2)2^{-m}\}$, then $U_m\subseteq E_{m}$ and $U_j\cap E_{m}=\emptyset$ for all $j\neq m$.
For each $m\in\N$, let $I_A(m)$ denote the set of all $k\in H_{A}(E_m)\setminus H_A(V(\mathcal C))$ and let $I_B(m)$ denote the set of all $k\in H_B(E_m)\setminus H_B(V(\mathcal C))$. Note that, $\rho_{E_m}(A)\leq\rho_{W_m}(A)<(3/2)\gamma_m2^{-m-2}$. Since $\varepsilon<2^{-5}$, it follows that, if $k\in H_{U_m}(A)$, we have
\[
\cb{a_k}{r_k}\subseteq\bigg\{z\in\C:\frac342^{-m}<\abs{z}<\frac542^{-m}\bigg\} \subseteq\tint{E_m}
\]
Hence $H_A(E_0)=I_A(0)$ and $H_B(E_0)=I_B(0)$ and, for all $m\geq 1$, we have $I_A(m)=H_A(E_m)\setminus H_A(U_{m})$ and $I_B(m)=H_B(E_m)\setminus H_B(U_{m})$. Now since $B\in\LM C$ we see that $\sum_{n\in I_B(m)}s_n=\sum_{n\in I_A(m)}r_n$, for all $m\in\N$, by Lemma \ref{diskscorrespondence}.
Hence $\rho_{E_0}(B)=\rho_{E_0}(A)$ and for $m\geq 1$ we have
\[
\rho_{E_m}(B)=\sum_{n\in I_B(m)}s_n+\rho_{U_{m}}(B)\leq\sum_{n\in I_A(m)}r_n+\rho_{U_{m}}(A)=\rho_{E_m}(A).
\]
Since $\rho_{E_m}(A)\leq\rho_{W_m}(A)<\gamma_m$, for each $m\in\N$, it now follows that
\[
\sum\limits_{m=1}^\infty m^m\rho_{E_m}(B)\leq\sum\limits_{m=1}^\infty m^m\frac{\varepsilon}{2^mm^m}=\varepsilon<\infty.
\]
Since each disk meets at most two of the $E_m$, $R(X_B)$ admits a non-degenerate bounded point derivation of infinite order at $0$ by Hallstrom's theorem \cite{hallstrom1969} (see also \cite{o1979regular}).
\end{proof}

We raise the following open question related to regularity and bounded point derivations.

\begin{question}
Let $X$ be a compact plane set such that $R(X)$ has no non-zero bounded point derivations. Is $R(X)$ necessarily regular$?$
\end{question}

We would like to thank the referee for helpful comments.

%\section*{References}

%\bibliographystyle{siam}
%\bibliography{bibliography}

\noindent
School of Mathematical Sciences, University of Nottingham\\
University Park, Nottingham, NG7 2RD, UK\\
E-mail: joel.feinstein@nottingham.ac.uk

\sskip
\noindent
School of Mathematical Sciences, University of Nottingham\\
University Park, Nottingham, NG7 2RD, UK\\
E-mail: pmxsm9@nottingham.ac.uk

\sskip
\noindent
School of Mathematical Sciences, University of Nottingham\\
University Park, Nottingham, NG7 2RD, UK\\
E-mail: pmxhy1@nottingham.ac.uk
\end{document}